\documentclass[12pt]{article}
\usepackage{amsmath,latexsym,amssymb}
\usepackage{enumerate}
\usepackage{amsthm}
\usepackage{epsfig,graphicx}

\newcommand\NN{\mathbb{N}}
\newcommand\RR{\mathbb{R}}
\newcommand\ZZ{\mathbb{Z}}
\newcommand\MM{\mathcal{M}}

\newcommand\FF{\mathcal{F}}
\newcommand\GG{\mathcal{G}}

\newcommand\PP{\mathcal{P}}

\newcommand\TT{\mathcal{T}}
\newcommand\UU{\mathcal{U}}
\newcommand\eps{{\varepsilon}}
\DeclareMathOperator\sP{P}   
\newcommand{\rP}{\mathrm{P}} 
\DeclareMathOperator\sE{E}   
\newcommand{\rE}{\mathrm{E}} 

\DeclareMathOperator\dist{dist}

\DeclareMathOperator\diam{diam}

\DeclareMathOperator\sgn{sgn}
\renewcommand{\mod}{\operatorname{mod}}
\DeclareMathOperator\graph{Graph}
\newcommand\rest{\!\!\downharpoonright}

\newtheorem{theorem}{Theorem}[section]

\newtheorem{corollary}[theorem]{Corollary}

\newtheorem{lemma}[theorem]{Lemma}
\newtheorem{remark}[theorem]{Remark}
\newtheorem{proposition}[theorem]{Proposition}

\numberwithin{equation}{section}

\newcommand{\address}{Address: Department of Mathematics, University of North Texas, 1155 Union Circle \#311430, Denton, TX 76203-5017, USA; E-mail: allaart@unt.edu}

\title{Hausdorff dimension of level sets of generalized Takagi functions}
\author{Pieter C. Allaart \footnote{\address}}

\begin{document}

\maketitle

\begin{abstract}

This paper examines the Hausdorff dimension of the level sets $f^{-1}(y)$ of continuous functions of the form
\begin{equation*}
f(x)=\sum_{n=0}^\infty 2^{-n}\omega_n(x)\phi(2^n x), \qquad 0\leq x\leq 1,
\end{equation*}
where $\phi(x)$ is the distance from $x$ to the nearest integer, and for each $n$, $\omega_n$ is a $\{-1,1\}$-valued function which is constant on each interval $[j/2^n,(j+1)/2^n)$, $j=0,1,\dots,2^n-1$. This class of functions includes Takagi's continuous but nowhere differentiable function. It is shown that the largest possible Hausdorff dimension of $f^{-1}(y)$ is $\log\big((9+\sqrt{105})/2\big)/\log 16\approx .8166$, but in case each $\omega_n$ is constant, the largest possible dimension is $1/2$. These results are extended to the intersection of the graph of $f$ with lines of arbitrary integer slope. Furthermore, two natural models of choosing the signs $\omega_n(x)$ at random are considered, and almost-sure results are obtained for the Hausdorff dimension of the zero set and the set of maximum points of $f$. The paper ends with a list of open problems.

\bigskip
{\it AMS 2000 subject classification}: 26A27, 28A78 (primary), 15A18, 15B48, 37H10 (secondary)

\bigskip
{\it Key words and phrases}: Generalized Takagi function, Gray Takagi function, Level set, Hausdorff dimension, Random Takagi function, Joint spectral radius, Random fractal
\end{abstract}

\section{Introduction}

Takagi's continuous nowhere differentiable function, shown in Figure \ref{fig:Takagi-and-Gray-Takagi}(a), is defined by
\begin{equation*}
T(x)=\sum_{n=0}^\infty \frac{1}{2^n}\phi(2^n x),
\end{equation*}
where $\phi(x)=\dist(x,\ZZ)$, the distance from $x$ to the nearest integer.
A great deal has been written about this function since its initial discovery in \cite{Takagi}; see recent surveys by Allaart and Kawamura \cite{AK2} and Lagarias \cite{Lagarias} for an overview of the literature. In the past few years, interest has focused mainly on the level sets $L(y):=\{x\in[0,1]:T(x)=y\}$, which have been shown to possess a rich structure. For instance, $L(y)$ is finite for Lebesgue-almost every $y$ \cite{Buczolich}, and can have any even positive integer as its cardinality \cite{Allaart2}. However, the ``typical" level set of $T$ (in the sense of Baire category) is uncountably large \cite{Allaart3,Allaart5}. These uncountable level sets can be further differentiated according to their Hausdorff dimension. Kahane \cite{Kahane} showed that $\max_x T(x)=2/3$, and $L(2/3)$ is a Cantor set of Hausdorff dimension $1/2$. De Amo et al. \cite{ABDF} recently proved that $1/2$ is the maximal Hausdorff dimension of any level set, settling a conjecture of Maddock \cite{Maddock}, who had earlier obtained a bound of $0.668$. Two interesting papers by Lagarias and Maddock \cite{LagMad1,LagMad2} use novel notions of `local level sets' and a `Takagi singular function' to establish several properties of the level sets of $T$. For instance, it is shown in \cite{LagMad2} that the set of $y$-values for which $L(y)$ has strictly positive Hausdorff dimension is a set of full Hausdorff dimension 1.

In this article we examine the level sets of a class of generalized Takagi functions of the form
\begin{equation}
f(x)=\sum_{n=0}^\infty\frac{\omega_n(x)}{2^n}\phi(2^n x), \qquad 0\leq x\leq 1,
\label{eq:our-functions}
\end{equation}
where
\begin{equation*}
\omega_n(x)\in\{-1,1\}, \quad\mbox{constant on each interval} \ \left[\frac{j}{2^n},\frac{j+1}{2^n}\right), \quad j=0,1,\dots,2^n-1.
\end{equation*}
Observe that $\omega_n$ can jump only at points $x$ where $\phi(2^n x)=0$, so the terms of the series in \eqref{eq:our-functions} are continuous. As a result, $f$ is a continuous function. We denote by $\TT_v$ the class of all functions $f$ of the form \eqref{eq:our-functions}, and by $\TT_c$ the subclass of those $f$ in $\TT_v$ for which $\omega_n$ is constant for each $n$. The class $\TT_v$ was investigated in detail by Abbott, Anderson and Pitt \cite{Abbott}, who denoted it by $\Lambda_{d,1}^*$ because of its relationship with Zygmund's class $\Lambda^*$ of quasi-smooth functions. But whereas \cite{Abbott} studies the class $\TT_v$ from the perspective of abscissa or $x$-values, our focus here is on ordinate or $y$-values.

Several members of $\TT_v$ have featured in the literature. These include the {\em alternating Takagi function} (e.g. \cite{Abbott,Kawamura}), for which $\omega_n=(-1)^n$ and which hence lies in $\TT_c$; and the {\em Gray Takagi function} \cite{Kobayashi}, for which $\omega_n(x)=r_n(x)$, where $r_n$ denotes the $n$th Rademacher function defined by
\begin{equation}
r_n(x):={(-1)}^{\lfloor 2^n x\rfloor}, \qquad n=0,1,2,\dots
\label{eq:Rademacher}
\end{equation}
The Gray Takagi function is shown in Figure \ref{fig:Takagi-and-Gray-Takagi}(b).
Another example is the function $T^3$ of Kawamura \cite{Kawamura}, which has $\omega_n(x)=r_1(x)\cdots r_n(x)$. All members of $\TT_v$ are nowhere differentiable; Billingsley's proof \cite{Billingsley} for the Takagi function extends easily. Furthermore, their graphs have Hausdorff dimension one \cite{Anderson}. All functions in the subclass $\TT_c$ are symmetric with respect to $x=1/2$, and their level set structure is similar to that of $T$: Lebesgue-almost all level sets are finite, but the ``typical" level set is uncountably infinite; see \cite{Allaart4}. Whether these properties hold for the wider class $\TT_v$ remains unsolved.

This paper concerns the Hausdorff dimension of the level sets
\begin{equation*}
L_f(y):=\{x\in[0,1]: f(x)=y\}, \qquad f\in\TT_v, \quad y\in\RR.
\end{equation*}
For a Borel set $E\subset\RR$, we denote the Hausdorff dimension of $E$ by $\dim_H E$.

\begin{figure}
\begin{center}
(a)\epsfig{file=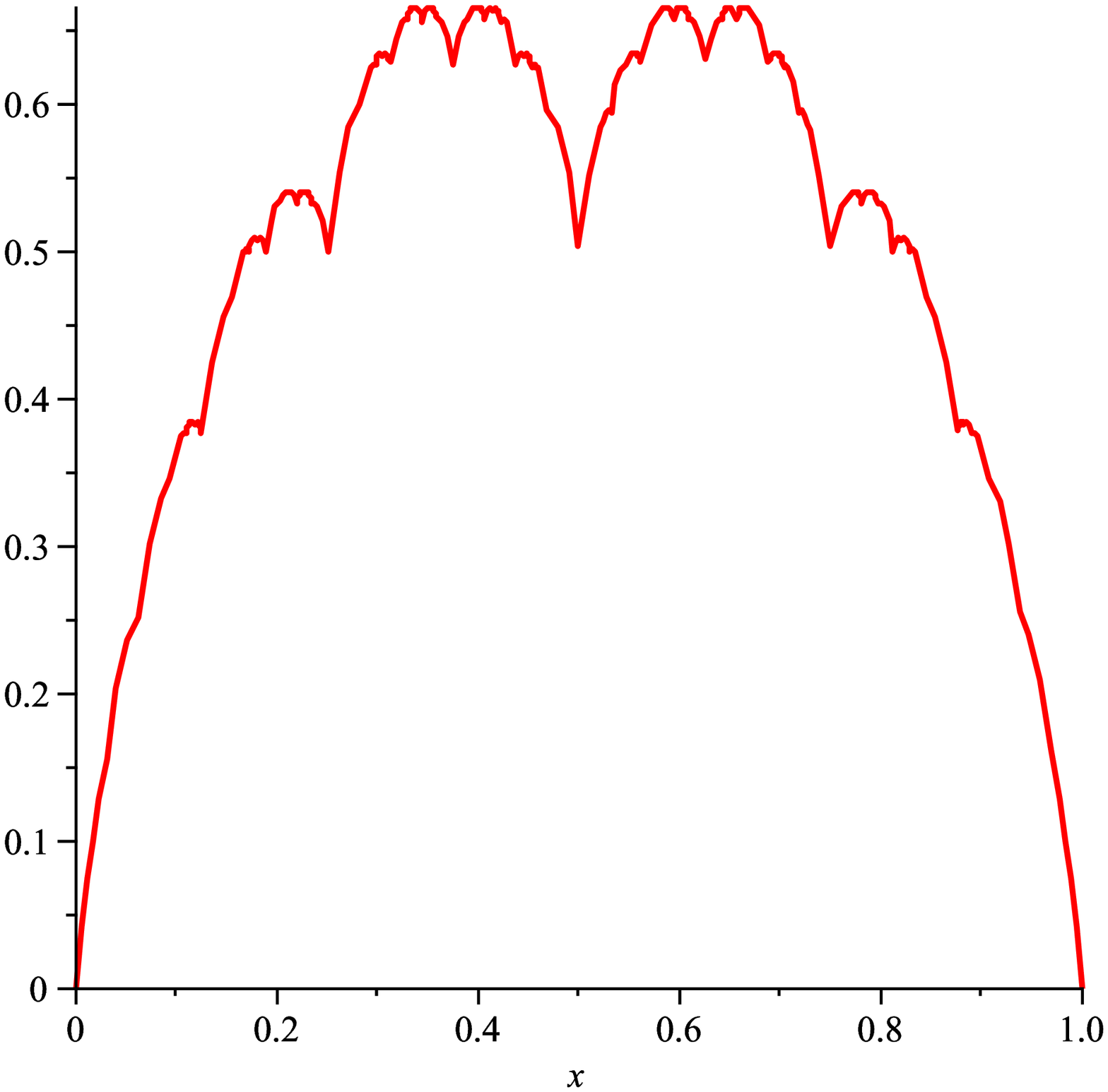, height=.18\textheight, width=.36\textwidth} \qquad
(b)\epsfig{file=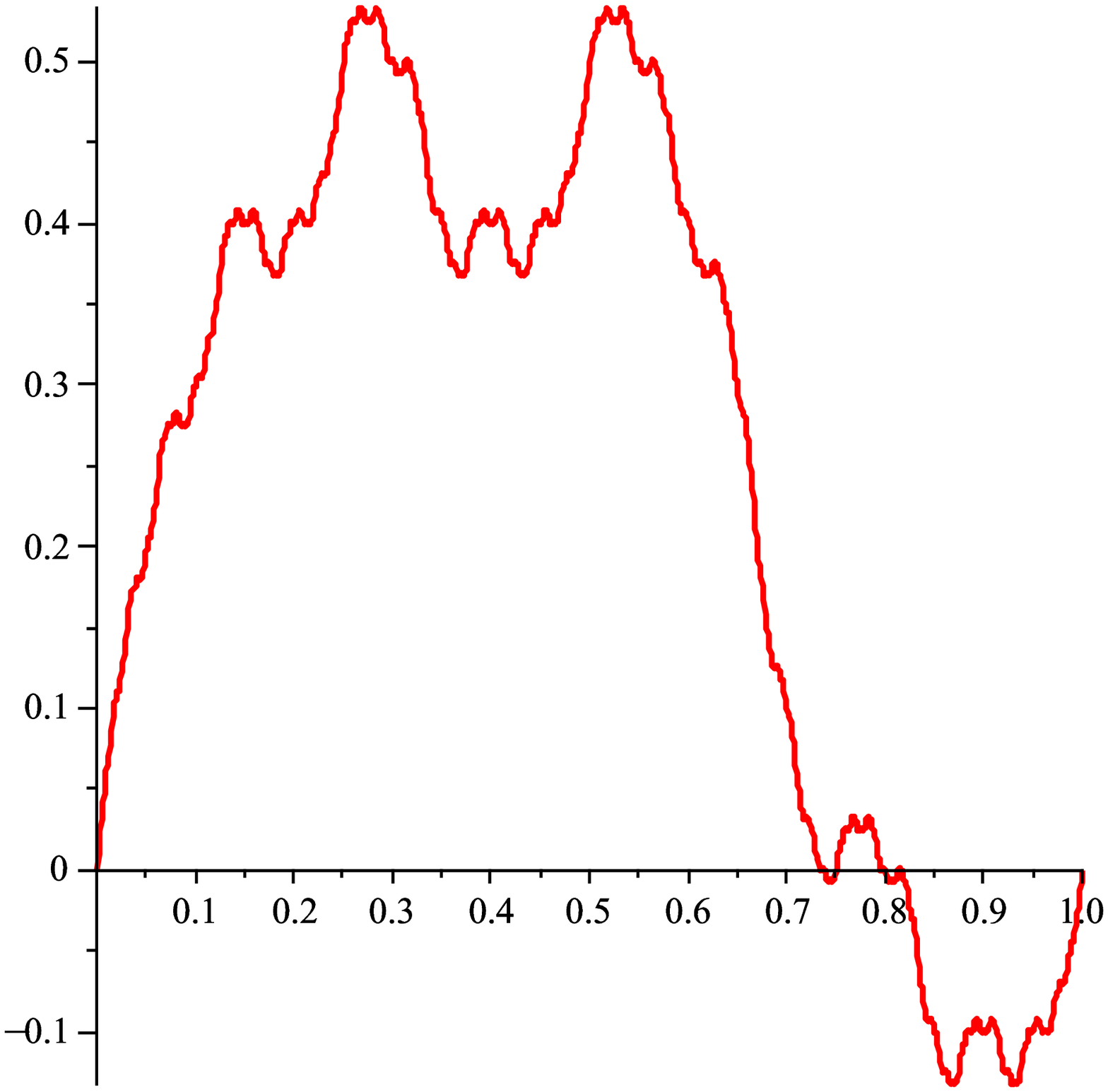, height=.18\textheight, width=.36\textwidth}
\caption{The Takagi function (left) and the Gray Takagi function}
\label{fig:Takagi-and-Gray-Takagi}
\end{center}
\end{figure}

\subsection{Sharp upper bounds} \label{subsec:sharp-bounds}

The first half of the paper gives sharp upper bounds for $\dim_H L_f(y)$, first for $f\in\TT_c$, and then for $f\in\TT_v$. 

\begin{theorem} \label{thm:rigid-upper-bound}
Let $f\in\TT_c$, with representation $f(x)=\sum_{n=0}^\infty 2^{-n}\omega_n\phi(2^n x)$. Then
\begin{equation}
\max_y \dim_H L_f(y)=1/2.
\label{eq:rigid-upper-bound}
\end{equation}
The maximum in \eqref{eq:rigid-upper-bound} is attained at a set of $y$-values dense in the range of $f$, and in particular, at 
\begin{equation}
y=\sum_{n=0}^\infty \frac{\omega_{2n}+\omega_{2n+1}}{4^{n+1}}.
\label{eq:rigid-extremal-level-set}
\end{equation}
\end{theorem}

The Gray Takagi function satisfies \eqref{eq:rigid-upper-bound} even though it does not belong to $\TT_c$.

\begin{theorem} \label{thm:Gray-upper-bound}
Let $f$ be the Gray Takagi function,
\begin{equation}
f(x)=\sum_{n=0}^\infty\frac{r_n(x)}{2^n}\phi(2^n x),
\label{eq:Gray-Takagi}
\end{equation}
where $r_n$ is the $n$th Rademacher function. Then $\max_y \dim_H L_f(y)=1/2$.
\end{theorem}

For general $f\in\TT_v$, however, the Hausdorff dimension of the level sets of $f$ can be much greater:

\begin{theorem} \label{thm:flexible-upper-bound}
Let $\alpha:=(9+\sqrt{105})/2$, and put $d_v^*:=\log\alpha/\log 16$. We have
\begin{equation}
\max_{f\in\TT_v} \max_{y\in\RR} \dim_H L_f(y)=d_v^*\approx .8166.
\label{eq:flexible-upper-bound}
\end{equation}
\end{theorem}

By embedding the graph of $f$ affinely into the graph of a suitable function $g\in \TT_v$, the above results can be extended to the intersection of the graph of $f$ with arbitrary lines of integer slope.

\begin{corollary} \label{cor:integer-slope}
For $m,b\in\RR$, let $\ell_{m,b}$ denote the line with equation $y=mx+b$.
\begin{enumerate}[(i)]
\item For each $f\in\TT_c$ and each $m\in\ZZ$, 
\begin{equation*}
\max_{b\in\RR} \dim_H (\graph(f)\cap \ell_{m,b})=1/2.
\end{equation*}
\item Let $d_v^*$ be as in Theorem \ref{thm:flexible-upper-bound}. For each $m\in\ZZ$,
\begin{equation*}
\max_{f\in\TT_v} \max_{b\in\RR} \dim_H (\graph(f)\cap \ell_{m,b})=d_v^*.
\end{equation*}
\end{enumerate}
\end{corollary}

We prove Theorems \ref{thm:rigid-upper-bound}, \ref{thm:Gray-upper-bound} and \ref{thm:flexible-upper-bound} and Corollary \ref{cor:integer-slope} in Section \ref{sec:upper-bounds}, by modifying and extending the method of De Amo et al. \cite{ABDF}. The idea is to consider the intersection of the graph of the partial sum $f_{2n}$ of $f$ (defined in \eqref{eq:partial-sum-functions} below) with a suitably chosen horizontal strip $[0,1]\times J_n$, where $\{J_n\}$ is a nested sequence of intervals shrinking to $\{y\}$. We then derive a system of linear recursions for the number of line segments in these strips (differentiated according to their slopes). An added complication is, that the coefficients in these recursions are dependent on $n$. Thus, for instance, the greater part of the proof of \eqref{eq:flexible-upper-bound} consists in determining the joint spectral radius of a certain pair of $3\times 3$ nonnegative matrices. It appears to be a lucky coincidence that this can be done exactly.

\subsection{The random case: zero sets and maximal sets}

Perhaps most interesting is the case when the signs $\omega_n$ are chosen at random; we consider two natural schemes here. In Model 2 below and in the rest of the paper, $\ZZ_+$ denotes the set of nonnegative integers. Let $(\Omega,\mathcal{F},\rP)$ be a probability space large enough to accomodate an infinite sequence of Bernoulli random variables with arbitrary success probabilities.

\bigskip
{\bf Model 1.} ({\em Random choice from $\TT_c$}) Let $\omega_0,\omega_1,\dots$ be independent and identically distributed (i.i.d.) random variables on $(\Omega,\mathcal{F},\rP)$ with $\rP(\omega_n=1)=p=1-\sP(\omega_n=-1)$, where $0<p<1$, and set $\omega_n(x):=\omega_n$.

\bigskip
{\bf Model 2.} ({\em Random choice from $\TT_v$}) Let $\{\omega_{n,j}: n\in\ZZ_+, j=0,1,\dots,2^n-1\}$ be i.i.d. random variables on $(\Omega,\mathcal{F},\rP)$ with $\rP(\omega_{n,j}=1)=p=1-\sP(\omega_{n,j}=-1)$, and set $\omega_n(x):=\omega_{n,j}$ if $x\in[j/2^n,(j+1)/2^n)$.

\bigskip
In either model, set $q:=1-p$. Since it seems difficult to treat the level sets in full generality, we focus on two special cases: the zero set and the set of maximum points of the random function $f$. Let
\begin{equation*}
d_0:=\frac{\log\big((1+\sqrt{5})/2\big)}{\log 4}\approx .3471.
\end{equation*}
For the zero set of $f$, we have the following results.

\begin{theorem} \label{thm:zero-set-rigid}
Assume Model 1. 
\begin{enumerate}[(i)]
\item The zero set $L_f(0)$ is finite with probability
\begin{equation*}
\rP(\#L_f(0)<\infty)=1-\min\{p/q,q/p\},
\end{equation*}
and given that $L_f(0)$ is infinite, $\dim_H L_f(0)>0$ a.s., for each $p\in(0,1)$.
\item If $p=1/2$, then $\dim_H L_f(0)\leq d_0$ a.s.
\end{enumerate}
\end{theorem}

\begin{theorem} \label{thm:zero-set-flexible}
Assume Model 2. If $p=1/2$, then $\dim_H L_f(0)=d_0$ a.s.
\end{theorem}

Which functions $f\in\TT_v$, specifically, satisfy $\dim_H L_f(0)=d_0$? Again the Gray Takagi function provides an example.

\begin{proposition} \label{prop:Gray-Takagi-zero-set}
Let $f$ be the Gray Takagi function. Then $\dim_H L_f(0)=d_0$.
\end{proposition}

Next, for $f\in\TT_v$, define
\begin{equation*}
M_f:=\max_{x\in[0,1]} f(x),
\end{equation*}
and
\begin{equation*}
\MM_f:=\{x\in[0,1]: f(x)=M_f\}=L_f(M_f).
\end{equation*}
Note that $f\leq T$ for all $f\in\TT_v$, so $M_f\leq 2/3$. The following theorem was proved in \cite{Allaart1} and is included here for comparison with Model 2.

\begin{theorem} \label{thm:maximal-set-rigid}
Assume Model 1. 
\begin{enumerate}[(i)]
\item If $p\geq 1/2$, then the distribution of $M_f$ is singular continuous and $\dim_H \MM_f=1-(2p)^{-1}$ a.s.
\item If $p<1/2$, then the distribution of $M_f$ is discrete and $\MM_f$ is finite a.s.
\end{enumerate}
\end{theorem}

In fact, the paper \cite{Allaart1} specifies the distributions of $M_f$ and the cardinality of $\MM_f$ in considerable detail under the assumption of Model 1. The analysis appears to be much harder for Model 2, and we describe here only what happens when $p>1/\sqrt{2}$. Note the contrast with the previous theorem.

\begin{theorem} \label{thm:maximal-set-flexible}
Assume Model 2.
\begin{enumerate}[(i)]
\item The probability that $M_f$ attains the maximum possible value of $2/3$ is
\begin{equation*}
\rP(M_f=2/3)=\max\left\{\frac{2p^2-1}{p^3},0\right\}.
\end{equation*}
\item If $p>1/\sqrt{2}$, then
\begin{equation*}
\dim_H \MM_f=\frac{\log(2p^2)}{\log 4} \qquad \mbox{a.s.},
\end{equation*}
and the distribution of $M_f$ is discrete, supported on the set
\begin{equation*}
\left\{\frac12\sum_{j\in\Delta}4^{-j}: \Delta\subset \ZZ_+, \ \#(\ZZ_+\backslash \Delta)<\infty\right\}.
\end{equation*}
\end{enumerate}
\end{theorem}

It seems plausible (by monotonicity considerations) that $\dim_H \MM_f=0$ a.s. for $p\leq 1/\sqrt{2}$, but the author has not been able to prove this.

The results for the random case are proved in Section \ref{sec:random-case}. Theorem \ref{thm:zero-set-flexible} is proved by casting the zero set as the attractor of a Mauldin-Williams random recursive construction and by using properties of hitting times in a symmetric simple random walk. The proof of Theorem \ref{thm:zero-set-rigid} requires a different approach. For the upper bound, we use the Perron-Frobenius theorem to construct a sequence of positive supermartingales which lead, via the Martingale Convergence Theorem, to successive upper bounds for $\dim_H L_f(0)$; we then prove that these upper bounds converge to $d_0$. For the lower bound we identify a particular combination of line segments, called a ``$Z$-shape", which will appear around the $x$-axis in the step-by-step construction of the graph of $f$ with probability one given that $L_f(0)$ is nonfinite. We show, via the law of large numbers, that the number of $Z$-shapes grows exponentially fast almost surely once a $Z$-shape appears. This, along with some additional observations, gives the lower bound. Finally, Theorem \ref{thm:maximal-set-flexible} is proved by considering a sequence of Galton-Watson branching processes associated with the random construction of $f$.

There are many natural questions still unanswered; Section \ref{sec:open} lists some of them.

\section{Preliminaries} \label{sec:prelim}

The following notation will be used throughout. For an interval $J$, $|J|$ denotes the diameter of $J$, and $J^\circ$ denotes the interior of $J$. The cardinality of a discrete set $\Lambda$ is denoted by $\#\Lambda$.
For $n\in\ZZ_+$ and $f$ defined by \eqref{eq:our-functions}, put
\begin{equation}
f_n(x):=\sum_{k=0}^{n-1}\frac{\omega_k(x)}{2^k}\phi(2^k x).
\label{eq:partial-sum-functions}
\end{equation}
Then $f_n$ is piecewise linear, and the right-hand derivative of $f_n$ at any point $x\in[0,1)$ is
\begin{equation*}
f_n^+(x)=\sum_{k=0}^{n-1}\omega_k(x)r_{k+1}(x),
\end{equation*}
where $r_n$ is defined as in \eqref{eq:Rademacher}.
Thus, $|f_{n+1}^+(x)-f_n^+(x)|=1$ for all $x$ and all $n$, and $f_{n+2}^+(x)-f_{n}^+(x)\in\{-2,0,2\}$. In particular, $f_{2n}^+(x)$ is always even.

Another important observation is that $f(k/2^n)=f_n(k/2^n)$ for integer $k$ and $n\in\ZZ_+$. We need two more elementary facts about the functions $f_n$:

\begin{lemma} \label{lem:tail-bound}
For each $n\in\ZZ_+$, $|f-f_n|<2^{-n}$.
\end{lemma}

\begin{proof}
We have the estimate
\begin{align*}
|f(x)-f_n(x)|&=\left|\sum_{k=n}^\infty \frac{\omega_k(x)}{2^k}\phi(2^k x)\right|\leq \sum_{k=n}^\infty 2^{-k}\phi(2^k x)\\
&=2^{-n}\sum_{m=0}^\infty 2^{-m}\phi(2^m 2^n x)=2^{-n}T(2^n x).
\end{align*}
Thus, the lemma follows from the bound $T(x)\leq 2/3$.
\end{proof}

\begin{lemma} \label{lem:even}
For $n\in\ZZ_+$ and integer $j$, $4^n f_{2n}(j/2^{2n})$ is an even integer.
\end{lemma}

\begin{proof}
We can write
\begin{equation}
4^n f_{2n}\left(\frac{j}{2^{2n}}\right)= \sum_{k=0}^{2n-1}\omega_k\left(\frac{j}{2^{2n}}\right)2^{2n-k}\phi\left(\frac{j}{2^{2n-k}}\right).
\label{eq:even-partial-sum}
\end{equation}
Now for each integer $m\geq 1$, $2^m\phi(j/2^m)$ is an integer with the same parity as $j$. To see this, note that $\phi(j/2^m)$ is either $j'/2^m$ or $(2^m-j')/2^m$, where $j'=j \mod 2^m$, and $j$, $j'$ and $2^m-j'$ all have the same parity. Thus, the terms in the right hand side of \eqref{eq:even-partial-sum} are either all even or all odd. Since there is an even number of them, their sum is even.
\end{proof}

We now introduce the closed intervals
\begin{gather*}
I_{n,j}:=\left[\frac{j}{2^n},\frac{j+1}{2^n}\right], \qquad n\in\ZZ_+, \quad j=0,1,\dots,2^n-1,\\
J_{n,k}:=\left[\frac{2k}{4^n},\frac{2k+2}{4^n}\right], \qquad n\in\ZZ_+,\ \ k\in\ZZ.
\end{gather*}
We denote by $\omega_{n,j}$ the value of $\omega_n$ on $I_{n,j}^\circ$, and by $s_{n,j}$ the slope of $f_n$ on $I_{n,j}$. Observe that
\begin{equation}
s_{n+1,2j}=s_{n,j}+\omega_{n,j}, \qquad s_{n+1,2j+1}=s_{n,j}-\omega_{n,j}.
\label{eq:slope-transition}
\end{equation}

\section{Universal upper bounds} \label{sec:upper-bounds}

In this section we prove the universal bounds of Subsection \ref{subsec:sharp-bounds}. After developing some notation, we prove Theorems \ref{thm:rigid-upper-bound} and \ref{thm:Gray-upper-bound} in Subsection \ref{subsec:Rigid-bounds}. The next two subsections together prove Theorem \ref{thm:flexible-upper-bound}. In Subsection \ref{subsec:extremal-function} we construct a function $f\in\TT_v$ and an ordinate $y$ for which $L_f(y)$ has Hausdorff dimension $d_v^*$, and in Subsection \ref{subsec:flexible-upper-bound} we show that $d_v^*$ is a universal upper bound for $\dim_H L_f(y)$. Subsection \ref{subsec:lines} gives a proof of Corollary \ref{cor:integer-slope}.

Following De Amo et al.~\cite{ABDF}, we divide the strip $[0,1]\times\RR$ into closed rectangles
$$R_{n,j,k}:=I_{2n,j}\times J_{n,k}, \qquad n\in\ZZ_+,\ 0\leq j<2^{2n},\ k\in\ZZ.$$
Let
\begin{gather*}
E_{n,k}:=\{j: R_{n,j,k}\ \mbox{intersects the graph of $f$}\},\\
E_{n,k}^0:=\left\{j: f_{2n}\equiv 2k/4^n\ \mbox{on}\ I_{2n,j}\right\},\\
E_{n,k}^1:=\{j: \mbox{the interior of $R_{n,j,k}$ intersects the graph of $f_{2n}$}\}.
\end{gather*}
Put
\begin{equation}
N_{n,k}:=\#E_{n,k}, \qquad N_{n,k}^0:=\#E_{n,k}^0, \qquad N_{n,k}^1:=\#E_{n,k}^1,
\label{eq:N-def}
\end{equation}
and
$$M_n:=\max_{k\in\ZZ} N_{n,k}, \qquad M_n^0:=\max_{k\in\ZZ} N_{n,k}^0, \qquad M_n^1:=\max_{k\in\ZZ} N_{n,k}^1.$$
Lemma \ref{lem:even} implies that if $j\in E_{n,k}^1$, then $R_{n,j,k}$ intersects the graph of $f_{2n}$ in a nonhorizontal line segment. Moreover, if $R_{n,j,k}$ intersects the graph of $f$, then by Lemma \ref{lem:tail-bound} the graph of $f_{2n}$ must intersect $R_{n,j,k-1}\cup R_{n,j,k}\cup R_{n,j,k+1}$. Hence, for each $k$,
\begin{equation*}
N_{n,k}\leq N_{n,k}^0+N_{n,k+1}^0+N_{n,k-1}^1+N_{n,k}^1+N_{n,k+1}^1,
\end{equation*}
so that
\begin{equation}
M_n\leq 2M_n^0+3M_n^1.
\label{eq:M-upper-estimate}
\end{equation}

\subsection{Rigid case: Proofs of Theorems \ref{thm:rigid-upper-bound} and \ref{thm:Gray-upper-bound}} \label{subsec:Rigid-bounds}

\begin{proof}[Proof of Theorem \ref{thm:rigid-upper-bound}]
Let $f\in\TT_c$. We claim that
\begin{equation}
M_{n+1}^0\leq \max\{2M_n^0,M_n^1\},
\label{eq:claim1}
\end{equation}
and
\begin{equation}
M_{n+1}^1\leq 2M_n^0+M_n^1.
\label{eq:claim2}
\end{equation}
To verify \eqref{eq:claim1} and \eqref{eq:claim2}, we consider four cases regarding $\omega_{2n}$ and $\omega_{2n+1}$. Let $\hat{{\bf s}}_{2n,j}:=(s_{2n+2,4j},s_{2n+2,4j+1},s_{2n+2,4j+2},s_{2n+2,4j+3})$.

\bigskip
{\em Case 1:} $\omega_{2n}=\omega_{2n+1}=1$. If $j\in E_{n,k}^0$, then $s_{2n,j}=0$ and hence $\hat{\bf s}_{2n,j}=(2,0,0,-2)$. Thus, $4j+1$ and $4j+2$ lie in $E_{n+1,4k+1}^0$, while $4j$ and $4j+3$ lie in $E_{n+1,4k}^1$. If $j\in E_{n,k}^1$, then $f_{2n+2}$ is monotone on $I_{2n,j}$. In particular, if $s_{2n,j}=2$, then $\hat{\bf s}_{2n,j}=(4,2,2,0)$ and hence, $4j+3\in E_{n+1,4k+4}^0$. Similarly, if $s_{2n,j}=-2$, then $4j\in E_{n+1,4k+4}^0$. If $|s_{2n,j}|\geq 4$, then no component of $\hat{\bf s}_{2n,j}$ is equal to $0$. We see that: 
\begin{itemize}
\item For $k'=4k+1$, each $j$ in $E_{n,k}^0$ contributes two members to $E_{n+1,k'}^0$ while each $j$ in $E_{n,k}^1$ contributes none. Hence
\begin{equation}
N_{n+1,4k+1}^0=2N_{n,k}^0.
\label{eq:case1-part1}
\end{equation}
\item For $k'=4k$, each $j$ in $E_{n,k}^0$ or $j$ in $E_{n,k}^1$ contributes no members to $E_{n+1,k'}^0$ while each $j$ in $E_{n,k-1}^1$ contributes at most one. Hence,
\begin{equation}
N_{n+1,4k}^0\leq N_{n,k-1}^1.
\label{eq:case1-part2}
\end{equation}
\item For $k'=4k+2$ or $k'=4k+3$, neither $j$ in $E_{n,k}^0$ nor $j$ in $E_{n,k}^1$ contributes to $E_{n+1,k'}^0$, so
\begin{equation}
N_{n+1,4k+2}^0=N_{n+1,4k+3}^0=0.
\label{eq:case1-part3}
\end{equation}
\end{itemize}
From \eqref{eq:case1-part1}, \eqref{eq:case1-part2} and \eqref{eq:case1-part3}, \eqref{eq:claim1} follows in Case 1. 

The verfication of \eqref{eq:claim2} is simpler: It is clear that each $j\in E_{n,k}^0$ contributes at most two members to any set $E_{n+1,k'}^1$ and each $j\in E_{n,k}^1$ contributes at most one by the monotonicity of $f_{2n+2}$ on $I_{2n,j}$. From this, \eqref{eq:claim2} follows.

\bigskip
{\em Case 2:} $\omega_{2n}=1$, $\omega_{2n+1}=-1$. Now if $j\in E_{n,k}^0$ we have $\hat{\bf s}_{2n,j}=(0,2,-2,0)$ and so $4j$ and $4j+3$ lie in $E_{n+1,4k}^0$, while $4j+1$ and $4j+2$ belong to $E_{n+1,4k}^1$. If $j\in E_{n,k}^1$ then $f_{2n+2}$ is monotone on $I_{2n,j}$, so $j$ contributes at most one member to any set $E_{n+1,k'}^1$. If in particular $s_{2n,j}=2$, then $\hat{\bf s}_{2n,j}=(2,4,0,2)$ and so $4j+2\in E_{n+1,4k+3}^0$. Similarly, if $s_{2n,j}=-2$, then $4j+1\in E_{n+1,4k+3}^0$. If $|s_{2n,j}|\geq 4$, then $\hat{\bf s}_{2n,j}$ does not have any zero components. We conclude that:
\begin{itemize}
\item For $k'=4k$, each $j$ in $E_{n,k}^0$ contributes two members to $E_{n+1,k'}^0$ while $j$ in $E_{n,k}^1$ or $j$ in $E_{n,k-1}^1$ contributes none. Hence
\begin{equation}
N_{n+1,4k}^0=2N_{n,k}^0.
\label{eq:case2-part1}
\end{equation}
\item For $k'=4k+3$, each $j$ in $E_{n,k}^0$ contributes no members to $E_{n+1,k'}^0$ and each $j$ in $E_{n,k}^1$ contributes at most one. Hence
\begin{equation}
N_{n+1,4k+3}^0\leq N_{n,k}^1.
\label{eq:case2-part2}
\end{equation}
\item Finally, for $k'=4k+1$ or $k'=4k+2$, neither $j$ in $E_{n,k}^0$ nor $j$ in $E_{n,k}^1$ contributes to $E_{n+1,k'}^0$, so
\begin{equation}
N_{n+1,4k+1}^0=N_{n+1,4k+2}^0=0.
\label{eq:case2-part3}
\end{equation}
\end{itemize}
From \eqref{eq:case2-part1}, \eqref{eq:case2-part2} and \eqref{eq:case2-part3}, \eqref{eq:claim1} follows in Case 2. And \eqref{eq:claim2} follows in the same way as in Case 1.

\bigskip
{\em Case 3:} $\omega_{2n}=-1, \omega_{2n+1}=1$. This case is similar to Case 2, by symmetry.

\bigskip
{\em Case 4:} $\omega_{2n}=\omega_{2n+1}=-1$. This case is similar to Case 1, by symmetry.

\bigskip

Using \eqref{eq:claim1}, \eqref{eq:claim2} and the initial conditions $M_0^0=1$ and $M_0^1=0$, it is straightforward to verify inductively that
\begin{equation*}
M_n^0\leq 2^n, \qquad M_n^1\leq 2(2^n-1).
\end{equation*}
Hence, by \eqref{eq:M-upper-estimate}, $M_n\leq 2^{n+1}+6(2^n-1)\leq 2^{n+3}$. It follows that for each $y$, $L_f(y)$ is covered by at most $2^{n+3}$ intervals of length $4^{-n}$. Thus,
\begin{equation*}
\dim_H L_f(y) \leq \lim_{n\to\infty}\frac{(n+3)\log 2}{n\log 4}=\frac12.
\end{equation*}

We next show that $y$ given by \eqref{eq:rigid-extremal-level-set} satisfies $\dim_H L_f(y)\geq 1/2$. 
The equalities in \eqref{eq:case1-part1} and \eqref{eq:case2-part1} show (by induction) that there is a sequence of ordinates $y_n:=2k_n/4^n, n\in\ZZ_+$, such that $N_{n,k_n}^0=2^n$ for each $n$. Specifically, $k_0=0$, and 
\begin{equation*}
k_{n+1}=\begin{cases}
4k_n+1, & \mbox{in Case 1},\\
4k_n, & \mbox{in Case 2 or 3},\\
4k_n-1, & \mbox{in Case 4}.
\end{cases}
\end{equation*}
Note that in particular, $k_{n+1}-4k_n=(\omega_{2n}+\omega_{2n+1})/2$.
Now let $K_n:=\bigcup\{I_{2n,j}:j\in E_{n,k_n}^0\}$.
The sets $\{K_n\}$ are compact and nested, so the intersection $K:=\bigcap_{n=0}^\infty K_n$ is a nonempty compact set. The intersection of $K_{n+1}$ with any interval $I:=I_{2n,j}$ in $K_n$ consists of two (non-overlapping) subintervals of one-fourth the length of $I$. Hence, by the basic theory of Moran fractals, $\dim_H K=\log 2/\log 4=1/2$. Now let $y:=\lim_{n\to\infty} y_n$, which exists since $|y_{n+1}-y_n|\leq 4^{-n}$. If $x\in K$, then $f_{2n}(x)=y_n$ for each $n$, and taking limits gives $f(x)=y$. Hence, $L_f(y)\supset K$, and so $\dim_H L_f(y)\geq 1/2$. One checks easily that the ordinate $y$ thus constructed is the one given by \eqref{eq:rigid-extremal-level-set}.

It remains to establish that the set $\{y:\dim_H L_f(y)=1/2\}$ is dense in the range of $f$. The key observation is that, when viewed as random variables on the probability space $[0,1)$ with Borel sets and Lebesgue measure, the slopes $f_n^+$, $n\in\ZZ_+$ form a symmetric simple random walk, so they take the value $0$ infinitely often with probability one. This implies that the sets
\begin{equation*}
E_m:=\bigcup_{n=m}^\infty\ \bigcup_{j: s_{n,j}=0} I_{n,j}, \qquad m\in\ZZ_+
\end{equation*}
have full measure in $[0,1]$, so they are in particular dense in $[0,1]$. Since $f$ is continuous, the set
\begin{equation*}
f(E_m)=\bigcup_{n=m}^\infty\ \bigcup_{j: s_{n,j}=0} f(I_{n,j})
\end{equation*}
is therefore dense in the range of $f$. Now if $s_{n,j}=0$, the restriction of the graph of $f$ to $I_{n,j}$ is a similar copy of the graph of some function $g\in \TT_c$, so by the previous part of the proof there is a $y\in f(I_{n,j})$ such that $\dim_H L_f(y)\geq 1/2$. It therefore suffices to show that for any subinterval $J$ of the range of $f$ there is a pair $(n,j)$ with $s_{n,j}=0$ for which $f(I_{n,j})\subset J$. Given such a $J$, choose $m\in\ZZ_+$ so large that $2^{1-m}<\frac12|J|$. Then if $n\geq m$ and $s_{n,j}=0$, we have
\begin{equation*}
\diam\big(f(I_{n,j})\big)\leq 2|I_{n,j}|=2^{1-n}\leq 2^{1-m}<\frac12|J|.
\end{equation*}
The denseness of $f(E_m)$ thus implies that at least one such interval $f(I_{n,j})$ must be fully contained in $J$. This completes the proof.
\end{proof}

\begin{proof}[Proof of Theorem \ref{thm:Gray-upper-bound}]
Let $f$ be the Gray Takagi function. Then $\omega_{n,j}=(-1)^j$ for all $n$ and $j$, and it is easy to use this along with \eqref{eq:slope-transition} to prove inductively that for each $n$, $s_{n,0}=n$ and $|s_{n,j+1}-s_{n,j}|=2$ for all $j$. This implies that $s_{n,j}\equiv n+2j\ (\mod 4)$ for all $n$ and all $j$. Thus, we have the following substitutions $s_{2n,j}\mapsto (s_{2n+2,4j},\dots,s_{2n+2,4j+3})$:

\bigskip
\begin{tabular}{lrlclrl}
$n$ even: & $0$ & $\mapsto (2,0,-2,0)$ & & $n$ odd: & $0$  &$\mapsto (0,-2,0,2)$\\
& $2$ & $\mapsto (2,0,2,4)$ &&& $2$ & $\mapsto (4,2,0,2)$\\
& $-2$ & $\mapsto (-2,-4,-2,0)$ \hspace{0.2in} &&& $-2$ & $\mapsto (0,-2,-4,-2)$.
\end{tabular}

\bigskip
Now we can infer that each horizontal line segment in the graph of $f_{2n}$ spawns two horizontal line segments at {\em different} levels in the graph of $f_{2n+2}$, and each line segment of slope $2$ in the strip $[0,1]\times J_{n,k}$ spawns one horizontal line segment at a {\em different} level than that spawned by a line segment of slope $-2$ in the same strip. Of course, if $|s_{2n,j}|\geq 4$, then $f_{2n+2}$ is strictly monotone on $I_{2n,j}$. Since $f_{2n}(0)=f_{2n}(1)=0$, exactly half the $N_{n,k}^1$ line segments intersecting the interior of the strip $[0,1]\times J_{n,k}$ have positive slope. From these observations, we conclude that
\begin{equation*}
M_{n+1}^0\leq M_n^0+\frac12 M_n^1, \qquad M_{n+1}^1\leq 2M_n^0+M_n^1.
\end{equation*}
With the initial values $M_0^0=1$ and $M_0^1=0$, it follows by induction that $M_n^0\leq 2^{n-1}$ and $M_n^1\leq 2^n$ for all $n\geq 1$. As in the proof of Theorem \ref{thm:rigid-upper-bound}, this implies that $\dim_H L_f(y)\leq 1/2$ for all $y$.

To complete the proof, we will show that $\dim_H L_f(2/5)=1/2$. Set $y_1:=1/2$ and $K_1:=[0,1/2]$. The left half of the graph of $f_2$ consists of two line segments with slopes $(2,0)$. Denote the union of these two line segments by $X$. In the graph of $f_4$, this figure is replaced by 8 line segments with slopes $(4,2,0,2,0,-2,0,2)$. See Figure \ref{fig:Gray-Takagi-substitution}. Observe that the third and fourth of these 8 line segments together form a similar copy of $X$, reduced by $1/4$ and rotated by $180^\circ$; and similarly for the seventh and eighth. Moreover, these two copies of $X$ sit at the same height in the graph of $f_4$, with their horizontal parts at height $y_2:=3/8$. We put $K_2:=[1/8,1/4]\cup[3/8,1/2]$. Next, in the transition from $f_4$ to $f_6$, each of these two rotated copies of $X$, now having slopes $(0,2)$, is replaced by a piecewise linear figure with slopes $(2,0,-2,0,2,0,2,4)$. This new figure contains two smaller copies of the original figure $X$, rotated back and sitting at equal height in the graph of $f_6$; see the second part of Figure \ref{fig:Gray-Takagi-substitution}. Thus, overall, the graph of $f_6$ contains (at least) four similar copies of $X$ with the same orientation, and with their horizontal part at height $y_3:=3/8+1/32$. Let $K_3$ be the (closure of) the union of the projections of these four copies of $X$ onto the $x$-axis.

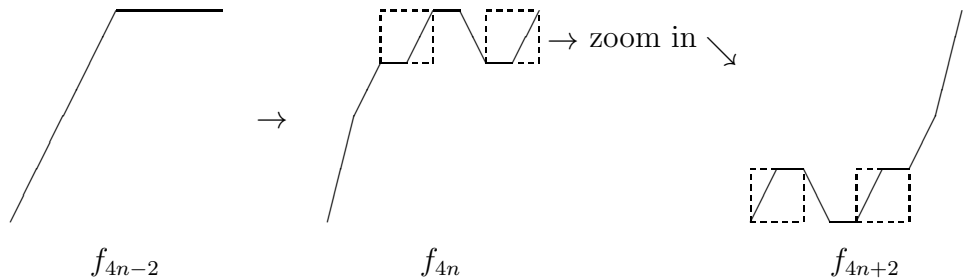
\begin{figure}
\begin{center}
\begin{picture}(360,120)(0,0)
\put(0,20){\line(1,2){40}}
\put(40,100){\line(1,0){40}}
\put(93,60){\makebox(0,0)[tl]{$\to$}}
\put(120,20){\line(1,4){10}}
\put(130,60){\line(1,2){10}}
\put(140,80){\line(1,0){10}}
\put(150,80){\line(1,2){10}}
\put(160,100){\line(1,0){10}}
\put(170,100){\line(1,-2){10}}
\put(180,80){\line(1,0){10}}
\put(190,80){\line(1,2){10}}
\multiput(140,80)(0,4){5}{\line(0,1){2}}
\multiput(140,80)(4,0){5}{\line(1,0){2}}
\multiput(160,80)(0,4){5}{\line(0,1){2}}
\multiput(140,100)(4,0){5}{\line(1,0){2}}
\multiput(180,80)(0,4){5}{\line(0,1){2}}
\multiput(180,80)(4,0){5}{\line(1,0){2}}
\multiput(200,80)(0,4){5}{\line(0,1){2}}
\multiput(180,100)(4,0){5}{\line(1,0){2}}
\put(203,90){\makebox(0,0)[tl]{$\to$}}
\put(219,94){\makebox(0,0)[tl]{zoom in}}
\put(263,90){\makebox(0,0)[tl]{$\searrow$}}
\put(280,20){\line(1,2){10}}
\put(290,40){\line(1,0){10}}
\put(300,40){\line(1,-2){10}}
\put(310,20){\line(1,0){10}}
\put(320,20){\line(1,2){10}}
\put(330,40){\line(1,0){10}}
\put(340,40){\line(1,2){10}}
\put(350,60){\line(1,4){10}}
\multiput(280,20)(4,0){5}{\line(1,0){2}}
\multiput(280,20)(0,4){5}{\line(0,1){2}}
\multiput(280,40)(4,0){5}{\line(1,0){2}}
\multiput(300,20)(0,4){5}{\line(0,1){2}}
\multiput(320,20)(4,0){5}{\line(1,0){2}}
\multiput(320,20)(0,4){5}{\line(0,1){2}}
\multiput(320,40)(4,0){5}{\line(1,0){2}}
\multiput(340,20)(0,4){5}{\line(0,1){2}}
\put(30,10){\makebox(0,0)[tl]{$f_{4n-2}$}}
\put(155,10){\makebox(0,0)[tl]{$f_{4n}$}}
\put(310,10){\makebox(0,0)[tl]{$f_{4n+2}$}}
\end{picture}
\end{center}
\caption{The substitution $(2,0)\mapsto (4,2,0,2,0,-2,0,2)$ when going from $f_{4n-2}$ to $f_{4n}$, and the substitution $(0,2)\mapsto (2,0,-2,0,2,0,2,4)$ when going from $f_{4n}$ to $f_{4n+2}$, in the construction of the Gray Takagi function.}
\label{fig:Gray-Takagi-substitution}
\end{figure}

It is not hard to see that this pattern continues: Each copy of $X$ in the graph of $f_{2n}$ induces two similar copies of $X$ in the graph of $f_{2n+2}$, reduced by $1/4$ and rotated $180^\circ$. Both copies sit at the same height in the graph of $f_{2n+2}$, with their horizontal part at level $y_{n+1}=(1/2)\sum_{i=0}^n(-1/4)^i$. Thus, by induction, the graph of $f_{2n+2}$ contains $2^{n}$ copies of $X$, equally oriented with their horizontal parts at level $y_{n+1}$. Let $K_{n+1}$ be the closure of the union of their projections onto the $x$-axis. 

The sets $\{K_n\}$ are nested; put $K:=\bigcap_{n=1}^\infty K_n$. Note that $y_n\to 2/5$. Now $x\in K$ implies that for each $n$, $|f_{2n}(x)-y_n|\leq 2(1/4)^{n}$, and letting $n\to\infty$ we get $f(x)=2/5$. Thus, $K\subset L_f(2/5)$. Finally, $K$ is a two-part generalized Cantor set with constant reduction ratio $1/4$, and hence, $\dim_H K=1/2$.
\end{proof}

\subsection{Flexible case: construction of an extremal function} \label{subsec:extremal-function}

We construct a function $f\in\TT_v$ and an ordinate $y$ such that $\dim_H L_f(y)\geq d_v^*$. We start with $f_0\equiv 0$, and describe inductively how to construct $f_{2n+2}$ from $f_{2n}$ for each $n$. At the same time, we construct a sequence of ordinates $\{y_n\}$, called {\em baselines}, which converge to the desired ordinate $y$. We also construct a nested sequence $\{K_n\}$ of compact subsets of $[0,1]$ such that everywhere on $K_n$, the slope of $f_{2n}$ takes values in $\{-2,0,2\}$ only. This sequence $\{K_n\}$ will converge to a compact limit set $K$, which will be a subset of $L_f(y)$ and will have Hausdorff dimension $d_v^*$. The idea of the construction is to build the successive approximants $f_{2n}$ in such a way as to maximize the asymptotic growth rate of the number of intervals $I_{2n,j}$ contained in the level set $L_{f_{2n}}(y_n)$.

Define the functions
\begin{equation*}
\phi_{n,j}(x):=\begin{cases} 2^{-n}\phi(2^n x), & \mbox{if $x\in I_{n,j}$}\\
0, & \mbox{otherwise}.
\end{cases}
\end{equation*}
Set $f_0\equiv 0$, $y_0=0$, and $K_0=[0,1]$. For $n\in\ZZ_+$, assume that $f_{2n}$, $y_n$ and $K_n$ have already been constructed in such a way that on $K_n$, the slope of $f_{2n}$ takes values in $\{-2,0,2\}$ only. Fix $j\in\{0,1,\dots,2^{2n}-1\}$, and note that on $I_{2n,j}$, $f_{2n+2}$ is obtained from $f_{2n}$ by adding 
$$\omega_{2n,j}\phi_{2n,j}+\omega_{2n+1,2j}\phi_{2n+1,2j}+\omega_{2n+1,2j+1}\phi_{2n+1,2j+1}.$$
Thus, the inductive construction of $f_{2n+2}$ from $f_{2n}$ on $I_{2n,j}$ is determined by the choice of the vector of signs ${\bf w}:=(\omega_{2n,j},\omega_{2n+1,2j},\omega_{2n+1,2j+1})$. If $I_{2n,j}$ does not lie in $K_n$, then it does not matter how we choose ${\bf w}$, but for definiteness, we put ${\bf w}=(1,1,1)$. Assume now that $I_{2n,j}\subset K_n$, and consider four cases. Let $s=s_{2n,j}$, so that $s\in\{-2,0,2\}$.

\bigskip
{\em Case 1:} $n\equiv 0\ (\mod 4)$. In this case we move the baseline up by setting $y_{n+1}=y_n+\frac12{(\frac14)}^n$. If $s=0$, we take ${\bf w}=(1,1,1)$; if $s=2$, we take ${\bf w}=(-1,1,-1)$; and if $s=-2$, we take ${\bf w}=(-1,-1,1)$.

\bigskip
{\em Case 2:} $n\equiv 1\ (\mod 4)$. In this case we do not move the baseline: $y_{n+1}=y_n$. If $s=0$, put ${\bf w}=(-1,1,1)$. If $s=\,\pm 2$ and $f_{2n}\leq y_n$ on $I_{2n,j}$, we put ${\bf w}=(1,1,1)$; otherwise, we put ${\bf w}=(-1,-1,-1)$.

\bigskip
{\em Case 3:} $n\equiv 2\ (\mod 4)$. In this case we move the baseline down by setting $y_{n+1}=y_n-\frac12{(\frac14)}^n$. If $s=0$, we take ${\bf w}=(-1,-1,-1)$; if $s=2$, we take ${\bf w}=(1,-1,1)$; and if $s=-2$, we take ${\bf w}=(1,1,-1)$.

\bigskip
{\em Case 4:} $n\equiv 3\ (\mod 4)$. In this case we do not move the baseline: $y_{n+1}=y_n$. If $s=0$, put ${\bf w}=(1,-1,-1)$. If $s=\,\pm 2$, we choose ${\bf w}$ the same way as in Case 2.

\bigskip
The transition from $f_{2n}$ to $f_{2n+2}$ is illustrated in Figure \ref{fig:substitutions} for Cases 1 and 2. Substitutions not shown in the figure are obtained by symmetry. Cases 3 and 4 are mirror images of Cases 1 and 2, respectively.

\begin{figure}
\begin{center}
\begin{picture}(380,230)(0,20)
\put(0,230){\makebox(0,0)[tl]{a)}}
\put(20,150){\line(1,0){40}}
\put(80,160){\makebox(0,0)[tl]{$\to$}}
\put(110,150){\line(1,2){10}}
\put(120,170){\line(1,0){20}}
\put(140,170){\line(1,-2){10}}
\put(113,160){\makebox(0,0)[br]{\footnotesize $2$}}
\put(148,160){\makebox(0,0)[bl]{\footnotesize $2$}}
\put(240,150){\line(1,2){40}}
\put(300,190){\makebox(0,0)[tl]{$\to$}}
\put(330,150){\line(1,2){10}}
\put(340,170){\line(1,0){10}}
\put(350,170){\line(1,2){10}}
\put(360,190){\line(1,4){10}}
\put(250,195){\makebox(0,0)[tl]{$2$}}
\put(333,159){\makebox(0,0)[br]{\footnotesize $2$}}
\put(363,177){\makebox(0,0)[br]{\footnotesize $2$}}
\put(373,205){\makebox(0,0)[br]{\footnotesize $4$}}
\qbezier[100](20,150)(195,150)(370,150)
\put(34,132){\makebox(0,0)[tl]{$f_{2n}$}}
\put(117,132){\makebox(0,0)[tl]{$f_{2n+2}$}}
\put(254,132){\makebox(0,0)[tl]{$f_{2n}$}}
\put(337,132){\makebox(0,0)[tl]{$f_{2n+2}$}}
\put(0,110){\makebox(0,0)[tl]{b)}}
\put(20,100){\line(1,0){40}}
\put(80,90){\makebox(0,0)[tl]{$\to$}}
\put(110,100){\line(1,0){10}}
\put(120,100){\line(1,-2){10}}
\put(130,80){\line(1,2){10}}
\put(140,100){\line(1,0){10}}
\put(117,92){\makebox(0,0)[tl]{\footnotesize $2$}}
\put(144,92){\makebox(0,0)[tr]{\footnotesize $2$}}
\put(240,100){\line(1,-2){40}}
\put(300,60){\makebox(0,0)[tl]{$\to$}}
\put(330,100){\line(1,0){10}}
\put(340,100){\line(1,-2){20}}
\put(360,60){\line(1,-4){10}}
\put(250,55){\makebox(0,0)[bl]{$2$}}
\put(353,88){\makebox(0,0)[br]{\footnotesize $2$}}
\put(363,70){\makebox(0,0)[br]{\footnotesize $2$}}
\put(373,37){\makebox(0,0)[br]{\footnotesize $4$}}
\qbezier[100](20,100)(195,100)(370,100)
\end{picture}
\end{center}
\caption{The substitutions of Case 1 (top) and Case 2 (bottom). The dotted line indicates the baseline. Numbers are absolute values of slopes.}
\label{fig:substitutions}
\end{figure}
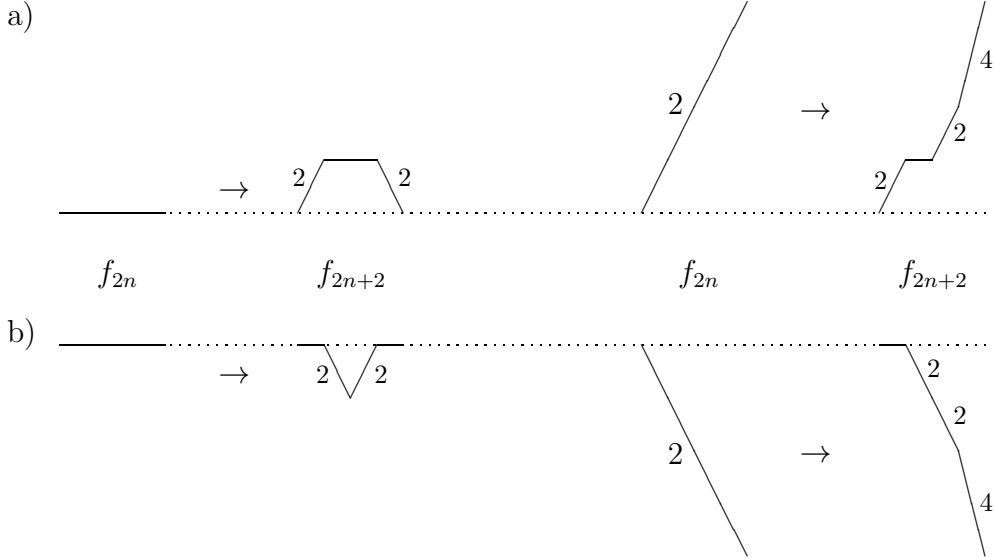

Now that $f_{2n+2}$ and $y_{n+1}$ have been constructed, we define $K_{n+1}$ as follows. For $j=0,1,\dots,2^{2n+2}-1$, put the interval $I_{2n+2,j}$ in $K_{n+1}$ if and only if $I_{2n+2,j}\subset K_n$ and $f_{2n+2}(x)=y_{n+1}$ for some $x\in I_{2n+2,j}$. Then on each interval $I_{2n+2,j}$ in $K_{n+1}$, the slope of $f_{2n+2}$ lies in $\{-2,0,2\}$.

This completes the construction of $\{f_{2n}\}, \{y_n\}$ and $\{K_n\}$. Now let $f:=\lim_{n\to\infty}f_{2n}$, $y:=\lim_{n\to\infty}y_n$, and $K:=\bigcap_{n=0}^\infty K_n$. Then $f\in \TT_v$, and it is not difficult to calculate that $y=8/17$. We claim that $K\subset L_f(y)$. To see this, let $x\in K$ and $n\in\NN$. The slope of $f_{2n}$ at $x$ (if defined) lies in $\{-2,0,2\}$, and there is a point $x_n$ with $|x_n-x|\leq 2^{-2n}$ and $f_{2n}(x_n)=y_n$. Hence, $|f_{2n}(x)-y_n|\leq 2^{1-2n}$. Letting $n\to\infty$ we obtain $f(x)=y$.

\begin{proposition}
We have $\dim_H K=d_v^*$, and consequently, $\dim_H L_f(y)\geq d_v^*$.
\end{proposition}

\begin{proof}
We shall represent $K$ as a multi-type Moran set. Recall that $K_n$ consists of those intervals $I_{2n,j}$ on which the graph of $f_{2n}$ either coincides with the $n$th baseline or else is a line segment with slope $\pm 2$ which has an endpoint on the baseline. Let $I:=I_{2n,j}$ be an interval in $K_n$ and let $s:=s_{2n,j}$, so $s\in\{-2,0,2\}$. We classify $I$ to be of type 1 if $s=0$ (so $f_{2n}\equiv y_n$ on $I$); of type 2 if $s=\pm\,2$ and $\sgn(f_{2n}-y_n)=\sgn(y_{n+2}-y_n)$ on $I^\circ$; and of type 3 if $s=\pm\,2$ and $\sgn(f_{2n}-y_n)=-\sgn(y_{n+2}-y_n)$ on $I^\circ$. Note that $y_{n+2}\neq y_n$ for each $n$, so this classification exhausts all possible cases. For illustration purposes, we extend the classification of types to the line segment which forms the graph of $f_{2n}$ over $I$. Thus, the line segments of slope $\pm\,2$ above the baseline are of type 2 in Cases 1 and 4 above, and of type 3 in Cases 2 and 3. For line segments below the baseline, the classification is reversed. 

The transition from $K_n$ to $K_{n+1}$ is summarized by a $3\times 3$ matrix $M^{(n)}=[m_{ij}^{(n)}]$, where $m_{ij}^{(n)}$ denotes the number of type $i$ intervals in $K_{n+1}$ contained in a type $j$ interval in $K_n$. Letting $A$ and $B$ be the matrices
\begin{equation*}
A:=\begin{bmatrix}2 & 1 & 0\\2 & 1 & 0\\0 & 1 & 0\end{bmatrix}, \qquad
B:=\begin{bmatrix}2 & 1 & 1\\2 & 1 & 0\\0 & 0 & 1\end{bmatrix},
\end{equation*}
we see from the four cases above that $M^{(n)}=A$ if $n$ is even, and $M^{(n)}=B$ if $n$ is odd. This suggests combining the stages two at a time and looking at the transition from $K_{2n}$ to $K_{2n+2}$, which is governed for each $n$ by the product matrix
\begin{equation*}
M:=BA=\begin{bmatrix}6 & 4 & 0\\6 & 3 & 0\\0 & 1 & 0\end{bmatrix}.
\end{equation*}
At this point we can eliminate type 3, because intervals of type 3 in $K_{2n}$ do not intersect $K_{2n+2}$ and hence do not intersect $K$. Thus, we may delete the last row and column of $M$, and conclude that $K$ is a two-type Moran set with construction matrix
\begin{equation*}
\hat{M}:=\begin{bmatrix}6 & 4\\6 & 3\end{bmatrix}.
\end{equation*}
The spectral radius of $\hat{M}$ is $\alpha=(9+\sqrt{105})/2$. Since each basic interval in $K_{2n+2}$ is $1/16$ the length of a basic interval in $K_{2n}$, the common contraction ratio is $1/16$. It follows from Marion \cite{Marion} that $\dim_H K=\log\alpha/\log 16=d_v^*$.
\end{proof}

The function $f$ constructed above is shown in Figure \ref{fig:extremal-function}, along with the horizontal line $y=8/17$ which intersects the graph of $f$ in a set of dimension $d_v^*$.

\begin{figure}
\begin{center}
\epsfig{file=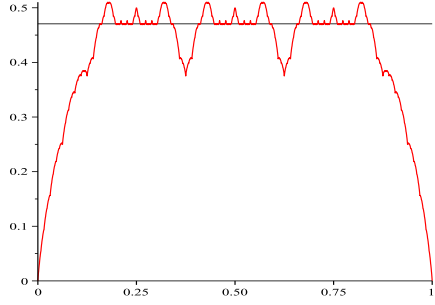, height=.2\textheight, width=.4\textwidth}
\caption{Graph of a function $f\in\TT_v$ attaining the maximal Hausdorff dimension $d_v^*$. The horizontal line represents the level $y=8/17$ for which $\dim_H L_f(y)=d_v^*$.}
\label{fig:extremal-function}
\end{center}
\end{figure}

\subsection{Flexible case: proof of the upper bound} \label{subsec:flexible-upper-bound}

To prove that $d_v^*$ is a universal upper bound for $\dim_H L_f(y)$, we again consider the numbers $N_{n,k}$, $N_{n,k}^0$ and $N_{n,k}^1$ defined in \eqref{eq:N-def}. In this setting, however, it is no longer adequate to examine the intersection of the graph of $f_{2n}$ with a single strip $[0,1]\times J_{n,k}$. Instead, we consider simultaneously two adjacent intervals $J_{n,k-1}$ and $J_{n,k}$, as outlined below.

Let $y\in\RR$ be fixed. For each $n$, $y$ belongs to some $J_{n,j}$, and we may assume that $y$ is not the right endpoint of $J_{n,j}$. If $2j/4^n\leq y<(2j+1)/4^n$, put $k_n=j$. Otherwise, if $(2j+1)/4^n\leq y<(2j+2)/4^n$, put $k_n=j+1$. Now define the numbers
\begin{equation*}
c_n:=N_{n,k_n}^0, \qquad l_n:=N_{n,k_n-1}^1, \qquad u_n:=N_{n,k_n}^1 \qquad (n\in\ZZ_+).
\end{equation*}
The dynamics of the triple $(c_n,l_n,u_n)$ depend in a complicated way on $y$ and $f$. To simplify the analysis, we introduce the new variables
\begin{equation*}
\sigma_n:=l_n+u_n, \qquad m_n:=\max\{l_n,u_n\},
\end{equation*}
and consider the dynamics of the vector ${\bf x}_n:=(c_n,\sigma_n,m_n)^t$, where $\cdot^t$ denotes matrix transpose. For matrices $A=[a_{ij}]$ and $B=[b_{ij}]$ of equal size, write $A\leq B$ if $a_{ij}\leq b_{ij}$ for all $(i,j)$. Let $E$ and $F$ be the matrices
\begin{equation*}
E:=\begin{bmatrix}2 & 0 & 1\\2 & 0 & 2\\2 & 0 & 1\end{bmatrix}, \qquad
F:=\begin{bmatrix}2 & 1 & 0\\2 & 1 & 0\\2 & 0 & 1\end{bmatrix}.
\end{equation*}

\begin{lemma} \label{lem:E-and-F}
For each $n$, either ${\bf x}_{n+1}\leq E{\bf x}_n$ or ${\bf x}_{n+1}\leq F{\bf x}_n$.
\end{lemma}

\begin{proof}
The choice of $k_n$ implies that $k_{n+1}\in\{4k_n-2,\dots,4k_n+2\}$. By symmetry, the argument for the case $k_{n+1}=4k_n-i$ is the same as that for $k_{n+1}=4k_n+i$ ($i=1,2$), so we have three essentially different cases to consider. Put $y_n:=2k_n/4^n$, so $y_n$ is the center of the combined interval $J_{n,k_n-1}\cup J_{n,k_n}$.

\bigskip
{\em Case 1:} $k_{n+1}=4k_n$. In this case, $y_{n+1}=y_n$. Each horizontal line segment in the graph of $f_{2n}$ at height $y_n$ contributes at most two horizontal line segments at height $y_{n+1}$ in the graph of $f_{2n+2}$, and each nonhorizontal line segment (whether above or below height $y_n$) contributes at most one such horizontal line segment. Thus, $c_{n+1}\leq 2c_n+\sigma_n$. Further, each horizontal line segment in the graph of $f_{2n}$ at height $y_n$ spawns two nonhorizontal line segments (of slopes $2$ and $-2$) either above or below height $y_{n+1}$, while each nonhorizontal line segment generates exactly one such segment. Thus, $\sigma_{n+1}\leq 2c_n+\sigma_n$. Finally, $u_{n+1}\leq 2c_n+u_n$ and $l_{n+1}\leq 2c_n+l_n$, so $m_{n+1}\leq 2c_n+m_n$. It follows that ${\bf x}_{n+1}\leq F{\bf x}_n$.

\bigskip
{\em Case 2:} $k_{n+1}=4k_n+1$. In this case, $y_{n+1}=y_n+\frac12{(\frac14)}^n$. Each horizontal line segment in the graph of $f_{2n}$ at height $y_n$ contributes at most two horizontal line segments to the graph of $f_{2n+2}$ at height $y_{n+1}$, while each non-horizontal line segment contributes at most one. Nonhorizontal line segments which lie entirely below $y_n$ contribute none. Thus, $c_{n+1}\leq 2c_n+u_n\leq 2c_n+m_n$. On the other hand, a horizontal line segment at height $y_n$ in the graph of $f_{2n}$ contributes at most two nonhorizontal line segments (of slope $2$ and $-2$, respectively) to the graph of $f_{2n+2}$ between heights $y_n$ and $y_{n+1}$, and none above height $y_{n+1}$. Finally, each nonhorizontal line segment in the graph of $f_{2n}$ that lies at least partially above $y_n$ contributes at most one nonhorizontal line segment to the graph of $f_{2n+2}$ in the strip $[0,1]\times J_{n+1,k_{n+1}}$, and at most one in the strip $[0,1]\times J_{n+1,k_{n+1}-1}$. Hence, $l_{n+1}\leq 2c_n+u_n$ and $u_{n+1}\leq u_n$. It follows that $\sigma_{n+1}\leq 2c_n+2u_n\leq 2c_n+2m_n$ and $m_{n+1}\leq 2c_n+m_n$. Thus, ${\bf x}_{n+1}\leq E{\bf x}_n$.

\bigskip
{\em Case 3:} $k_{n+1}=4k_n+2$. This is the simplest case, because here $c_{n+1}=0$, $l_{n+1}\leq u_n$ and $u_{n+1}\leq u_n$. Hence, $\sigma_{n+1}\leq 2u_n\leq 2m_n$ and $m_{n+1}\leq u_n\leq m_n$. It follows that ${\bf x}_{n+1}\leq E{\bf x}_n$.
\end{proof}

As a side note, we can observe from Case 3 above that if $y$ is such that $k_{n+1}=4k_n\pm 2$ for all but finitely many $n$, then $L_f(y)$ is finite. This will be the case if $y/2$ has quarternary expansion ending in all $2$'s, i.e. if $y=\pm(k+1/3)4^{-m}$ for some $k\in\ZZ_+$ and $m\in\ZZ_+$. In particular, if $y=\pm\, 1/3$, then $L_f(y)$ is either empty or consists of exactly two points, for each $f\in\TT_v$.

It transpires from Lemma \ref{lem:E-and-F} that we must compute the {\em joint spectral radius} of $E$ and $F$. For a set $\Sigma$ of $d\times d$ matrices, the joint spectral radius $\rho(\Sigma)$ is defined by
\begin{equation*}
\rho(\Sigma):=\limsup_{n\to\infty} \sup\{\|T_1 T_2\cdots T_n\|^{1/n}: T_i\in\Sigma\ \mbox{for all $i$}\}.
\end{equation*}
The value of $\rho(\Sigma)$ is independent of the choice of matrix norm $\|.\|$; see Rota and Strang \cite{Rota-Strang}. In general the joint spectral radius is difficult to compute exactly, even for sets of just two matrices. While much work has been done for the $2\times 2$ case (see M\"ossner \cite{Mossner} and the references therein), there are few known examples for larger matrices. For our set $\Sigma=\{E,F\}$, however, it is possible to show through a sequence of steps that $\rho(\Sigma)=\sqrt{\alpha}$, and this will give the upper bound in Theorem \ref{thm:flexible-upper-bound}.

\begin{proposition} \label{prop:joint-spectral-radius}
The joint spectral radius of $E$ and $F$ is given by
\begin{equation*}
\rho(\{E,F\})=\sqrt{\rho(FE)}=\sqrt{\alpha},
\end{equation*}
where for any square matrix $M$, $\rho(M)$ denotes the spectral radius of $M$.
\end{proposition}

\begin{proof}
We use the matrix norm $\|T\|=\sum_{i,j}|t_{ij}|$, where $T=[t_{ij}]$. If $T$ is a $3\times 3$ matrix with nonnegative entries, we have the representation $\|T\|={\bf 1}T{\bf 1}^t$, where ${\bf 1}:=(1,1,1)$. Denote by $\PP_k$ the set of all products $T_1\cdots T_k$ where $T_i\in\{E,F\}$ for $i=1,\dots,k$. Let $\PP_0:=\{I\}$, and put $\PP:=\bigcup_{k=0}^\infty \PP_k$.
A straightforward induction argument yields the following implications for $M\in\PP$:
\begin{gather}
(v_1,v_2,v_3)={\bf 1}M \qquad\Rightarrow\qquad v_1\geq\max\{v_2,v_3\},
\label{eq:left-multiply-by-1}\\
(w_1,w_2,w_3)^t=M{\bf 1}^t \qquad\Rightarrow\qquad \min\{w_1,w_2\}\geq w_3.
\label{eq:right-multiply-by-1}
\end{gather}
Since
\begin{equation*}
E^2=\begin{bmatrix}6 & 0 & 3\\8 & 0 & 4\\6 & 0 & 3\end{bmatrix}, \qquad
EF=\begin{bmatrix}6 & 2 & 1\\8 & 2 & 2\\6 & 2 & 1\end{bmatrix},
\end{equation*}
we see by \eqref{eq:right-multiply-by-1} that $E^2 M{\bf 1}^t\leq EFM{\bf 1}^t$ for each $M\in\PP$. Thus, if $E^2$ occurs anywhere in the product $T_1 T_2\cdots T_n$, we can replace it with $EF$ without decreasing the norm of the product. Hence it suffices to consider products of the form
\begin{equation}
F^{j_1}E F^{j_2}E\cdots F^{j_k}E, \qquad k\in\NN, \quad j_1\geq
1,\dots,j_k\geq 1.
\label{eq:alternating-products}
\end{equation}

\begin{lemma} \label{lem:sparse-matrix}
Let $S$ be a matrix of the form
\begin{equation*}
S=\begin{bmatrix}a & 0 & 0\\a & 0 & 0\\0 & 0 & -b\end{bmatrix}, \qquad a\geq b\geq 0.
\end{equation*}
Then for all $M_1,M_2\in\PP$,
\begin{equation*}
{\bf 1}M_1 SM_2{\bf 1}^t\geq 0.
\end{equation*}
\end{lemma}

\begin{proof}
Let ${\bf v}=(v_1,v_2,v_3)={\bf 1}M_1$, and ${\bf w}=(w_1,w_2,w_3)^t=M_2{\bf 1}^t$. Then \eqref{eq:left-multiply-by-1} and \eqref{eq:right-multiply-by-1} imply that $v_1\geq v_3$ and $w_1\geq w_3$. Since $v_i,w_i\geq 0$, it follows that
\begin{equation*}
{\bf 1}M_1 SM_2{\bf 1}^t={\bf v}S{\bf w}=av_1 w_1+av_2 w_1-bv_3 w_3\geq av_1 w_1-bv_3 w_3\geq 0,
\end{equation*}
as desired.
\end{proof}

We now calculate further:
\begin{gather*}
FE=\begin{bmatrix}6 & 0 & 4\\6 & 0 & 4\\6 & 0 & 3\end{bmatrix}, \qquad
F^2 E=\begin{bmatrix}18 & 0 & 12\\18 & 0 & 12\\18 & 0 & 11\end{bmatrix}, \qquad
F^3 E=\begin{bmatrix}54 & 0 & 36\\54 & 0 & 36\\54 & 0 & 35\end{bmatrix},\\
(FE)^2=\begin{bmatrix}60 & 0 & 36\\60 & 0 & 36\\54 & 0 & 33\end{bmatrix}, \qquad
(FE)^2-F^3 E=\begin{bmatrix}6 & 0 & 0\\6 & 0 & 0\\0 & 0 & -2\end{bmatrix}.
\end{gather*}
Applying Lemma \ref{lem:sparse-matrix} to $S=(FE)^2-F^3 E$ yields
\begin{equation*}
{\bf 1}M_1(FE)^2 M_2{\bf 1}^t\geq {\bf 1}M_1(F^3 E)M_2{\bf 1}^t
\end{equation*}
for all $M_1,M_2\in\PP$, and hence, if $F^3 E$ occurs in $T_1 T_2\cdots T_n$, we can replace it with $(FE)^2$ without decreasing the norm of the product. Repeating this as many times as needed we can thus reduce the problem to products of the form \eqref{eq:alternating-products} with $j_l\in\{1,2\}$ for $l=1,\dots,k$.

It remains to eliminate factors $F^2 E$. Without loss of generality we may assume that the product $T_1 T_2 \cdots T_n$ contains an even number of such factors. (Otherwise, we can left multiply $T_1 T_2 \cdots T_n$ by $F^2 E$ to create an even number of occurrences; the extra factor has no impact in the limit as $n\to\infty$.) Note that two consecutive occurrences of $F^2 E$ are separated by some power of $FE$.

\begin{lemma}
For each $k\geq 0$, the matrix
\begin{equation*}
S_k:=(FE)^{k+3}-F^2 E(FE)^k F^2 E
\end{equation*}
satisfies the hypothesis of Lemma \ref{lem:sparse-matrix}.
\end{lemma}

\begin{proof}
Let $G=FE$. The characteristic polynomial of $G$ is $\lambda^3-9\lambda^2-6\lambda$, so $G$ has three distinct real eigenvalues and $\rho(G)=\alpha$. Furthermore, $G$ satisfies its own characteristic equation: 
\begin{equation}
G^3-9G^2-6G=0. 
\label{eq:G-equation}
\end{equation}
The key to showing that $S_k$ has the required form is the apparent coincidence that $F^2 E$ is ``almost" equal to $3FE$. Precisely, $F^2 E=3FE+D$, where $D=[d_{ij}]$ is the matrix whose only nonzero entry is $d_{33}=2$. Multiplying \eqref{eq:G-equation} by $G^k$ and putting the results together we obtain, after some elementary algebra,
\begin{equation}
S_k=6G^{k+1}-3G^{k+1}D-3DG^{k+1}-DG^k D.
\label{eq:Sk-expanded}
\end{equation}
Writing
\begin{equation*}
G^k=\begin{bmatrix}\alpha_k & 0 & \beta_k\\ \alpha_k & 0 & \beta_k\\ \gamma_k & 0 & \delta_k\end{bmatrix},
\end{equation*}
\eqref{eq:Sk-expanded} becomes
\begin{equation*}
S_k=\begin{bmatrix}6\alpha_{k+1} & 0 & 0\\6\alpha_{k+1} & 0 & 0\\0 & 0 & -6\delta_{k+1}-4\delta_k\end{bmatrix}.
\end{equation*}
It remains to verify that $6\alpha_{k+1}\geq 6\delta_{k+1}+4\delta_k$. But this follows easily by induction since $\{\alpha_k\}$ and $\{\delta_k\}$ both satisfy the recursion $x_{k+1}=9x_k+6x_{k-1}$ for $k\geq 2$ in view of \eqref{eq:G-equation}, and the inequality clearly holds for $k=0,1,2$.
\end{proof}

In view of the last lemma, we can replace $F^2 E(FE)^k F^2 E$ with $(FE)^{k+3}$ without decreasing the norm of the product $T_1 T_2\cdots T_n$. Applying this repeatedly we eliminate all factors $F^2 E$ two at a time. Therefore, the extremal case is $T_1 T_2\cdots T_n=(FE)^{n/2}$ (assuming without loss of generality that $n$ is even), and this shows that $\rho(\{E,F\})=\sqrt{\rho(FE)}=\sqrt{\alpha}$, proving the Proposition.
\end{proof}

\bigskip

We can now complete the proof of Theorem \ref{thm:flexible-upper-bound}. By Lemma \ref{lem:E-and-F} and Proposition \ref{prop:joint-spectral-radius} (or rather, its proof), there is a constant $C$ such that $c_n+\sigma_n\leq C\alpha^{n/2}$ for all $n$. Using Lemma \ref{lem:tail-bound} it follows that for each $n$, the level set $L_f(y)$ can be covered by at most $C\alpha^{n/2}$ intervals of length $4^{-n}$. Hence,
\begin{equation*}
\dim_H L_f(y)\leq \lim_{n\to\infty}\frac{\log(C\alpha^{n/2})}{-\log(4^{-n})}=\frac{\log\alpha}{\log 16}=d_v^*,
\end{equation*}
giving the stated upper bound.

\subsection{Intersection with lines of integer slope} \label{subsec:lines}

We end this section with a proof of Corollary \ref{cor:integer-slope}. It uses the following simple lemma:

\begin{lemma} \label{lem:all-integer-slopes}
For any $f\in\TT_v$ and for any $m\in\ZZ$, the partial function $f_{|m|}$ has slope $m$ on exactly one interval $I_{|m|,j}$ with $0\leq j<2^{|m|}$.
\end{lemma}

\begin{proof}
We prove the statement for $m\geq 0$; the case $m<0$ follows by symmetry.
The statement is obvious for $m=0$. Proceeding by induction, suppose $s_{m,j}=m$. Then by \eqref{eq:slope-transition}, $s_{m+1,2j}=m+1$ if $\omega_{m,j}=1$, while $s_{m+1,2j+1}=m+1$ if $\omega_{m,j}=-1$. The uniqueness of $j$ follows easily by induction as well.
\end{proof}

\begin{proof}[Proof of Corollary \ref{cor:integer-slope}]
For convenience we regard each $f\in\TT_v$ as a $1$-periodic function defined on all of $\RR$. However, we keep the convention that $\graph(f)=\{(x,f(x)):0\leq x\leq 1\}$.
Fix $m$ and $b$, and put $\ell:=\ell_{m,b}$. We first show that
\begin{equation}
\dim_H (\graph(f)\cap\ell)\leq \begin{cases}
1/2, & \mbox{if $f\in\TT_c$},\\
d_v^*, & \mbox{if $f\in\TT_v$}.
\end{cases}
\label{eq:general-line-upper-bound}
\end{equation}
Assume first that $m\geq 0$. Let
\begin{equation*}
g(x):=-\sum_{k=0}^{m-1}\frac{1}{2^k}\phi(2^k x)+\frac{1}{2^m}f(2^m x), \qquad 0\leq x\leq 1,
\end{equation*}
and define the linear mapping $\Psi:\RR^2\to\RR^2$ by $\Psi(x,y):=(2^{-m}x,-m2^{-m}x+2^{-m}y)$. Then $\Psi$ maps $\graph(f)$ onto $\graph(g\rest_{[0,2^{-m}]})$, and it maps $\ell$ onto the horizontal line $\ell': y=b/2^m$. Since $\Psi$ is linear and invertible it is bi-Lipschitz, and hence
\begin{equation*}
\dim_H(\graph(f)\cap\ell)=\dim_H\big(\graph(g\rest_{[0,2^{-m}]})\cap\ell'\big) \leq \dim_H L_g(b/2^m).
\end{equation*}
Observe that $g\in\TT_v$, and if $f\in\TT_c$ then $g\in\TT_c$. Thus, \eqref{eq:general-line-upper-bound} follows for the case $m\geq 0$ from the bounds of Theorems \ref{thm:rigid-upper-bound} and \ref{thm:flexible-upper-bound}. And for $m<0$, it follows by applying the above argument to $-f$.

We next show that the upper bounds are attained for each $m\in\ZZ$. We do this for $m\geq 0$; the case $m<0$ follows by symmetry. First, let $f\in\TT_c$. By Lemma \ref{lem:all-integer-slopes} there is $0\leq j<2^{m}$ such that $f_{m}$ has slope $m$ on $I_{m,j}$. Let $x_0=j/2^m$ be the left endpoint of $I_{m,j}$. The function $h(x):=2^m\big(f(2^{-m}x+x_0)-f_m(2^{-m}x+x_0)\big)$ is in $\TT_c$, so its graph intersects some horizontal line $\ell$ in a set of dimension $1/2$. Let $\Phi:\RR^2\to\RR^2$ be the affine mapping that maps $\graph(h)$ onto $\graph(f\rest_{I_{m,j}})$. Then $\Phi$ maps $\ell$ onto a line $\ell'$ with slope $m$, and since $\Phi$ is bi-Lipschitz, it follows that
\begin{equation*}
\dim_H (\graph(f)\cap\ell')\geq\dim_H(\graph(f\rest_{I_{m,j}})\cap \ell')=\dim_H(\graph(h)\cap\ell)=1/2.
\end{equation*}

Finally, we build a function $g\in\TT_v$ which attains the bound in \eqref{eq:general-line-upper-bound}. Assume again that $m\geq 0$. Let $h\in\TT_v$ and $y_0\in\RR$ be such that $\dim_H L_h(y_0)=d_v^*$. Define $g\in\TT_v$ by
\begin{equation*}
g(x):=\sum_{k=0}^{m-1}\frac{1}{2^k}\phi(2^k x)+\frac{1}{2^m}h(2^m x), \qquad 0\leq x\leq 1.
\end{equation*}
Let $\Phi$ be the affine mapping that maps 
$\graph(h)$ onto $\graph(g\rest_{[0,2^{-m}]})$. Then $\Phi$ maps the line $y=y_0$ onto a line $\ell'$ of slope $m$, so that
\begin{equation*}
\dim_H (\graph(g)\cap\ell')\geq\dim_H\big(\graph(g\rest_{[0,2^{-m}]})\cap \ell'\big)=\dim_H L_h(y_0)=d_v^*.
\end{equation*}
This completes the proof.
\end{proof}

\section{The random case} \label{sec:random-case}

In this section we prove the results for the random case.

\subsection{Dimension of the zero set}

We first state a useful fact, which will be referred to as the {\em zero criterion}.

\bigskip
{\bf Zero criterion:} If $f_n$ does not take the value $0$ anywhere on an interval $I_{n,j}$, then $f$ itself will not vanish anywhere in $I_{n,j}$.

\bigskip
The zero criterion holds since $f_n>0$ on $I_{n,j}$ implies $f_n\geq 2^{-n}$ on $I_{n,j}$, and $|f-f_n|<2^{-n}$ by Lemma \ref{lem:tail-bound}. (A similar argument applies of course when $f_n<0$ on $I_{n,j}$.) The zero criterion implies that for each level $n$, we need only consider intervals $I_{n,j}$ on which the graph of $f_n$ has at least one point on the $x$-axis.

\begin{proof}[Proof of Theorem \ref{thm:zero-set-flexible}]
Assume $f$ is randomly generated according to Model 2 with $p=1/2$.
We calculate the almost-sure dimension of $L_f(0)\cap[0,1/2]$. By symmetry, this will equal the almost-sure dimension of $L_f(0)$. The idea is to represent $L_f(0)\cap[0,1/2]$ as the attractor of a Mauldin-Williams random recursive construction; see \cite{MW}. 
When $p=1/2$, the slope of $f_n$ on $I_{n,0}=[0,2^{-n}]$ follows a symmetric simple random walk, and as such it returns to zero infinitely often with probability one. This must happen at even times. If it happens for the first time at time $2n$, then $f_{2n}\equiv 0$ on $I_{2n,0}$, the slope of $f_{2n}$ on the adjacent interval $I_{2n,1}$ must be $\pm 2$, and by the zero criterion, $f$ cannot take the value $0$ anywhere else in $[0,1/2]$. Since Hausdorff dimension is independent of scale and orientation, we may assume without loss of generality that $n=1$, and the slope of $f_2$ on $I_{2,1}$ is $2$. We put $J_\emptyset:=[0,1/2]$, and call the graph of $f_2$ on $[0,1/2]$ our ``basic figure".

From here, the graph of $f$ will evolve independently on the intervals $[0,1/4]$ and $[1/4,1/2]$. More precisely, the restrictions of $f-f_2$ to these intervals are independent, and similarly, the restrictions of $f-f_3$ to the intervals $[0,1/8]$ and $[1/8,1/4]$ are independent. Thus, after random waiting times $T_1,T_2$ and $T_3$, independent of each other, the basic figure will reappear at smaller scale (and possibly reflected in the $x$-axis, reflected left-to-right, or both) just to the right of $0$, just to the left of $1/4$ and just to the right of $1/4$, respectively.
Here $T_1$ and $T_2$ have the distribution of $1+\tau_1$, and $T_3$ has the distribution of $\tau_2$, where for $m\in\NN$, $\tau_m$ is the hitting time of the random walk of level $m$. For $i=1,2,3$, let $J_i$ denote the projection of the $i$th copy of the basic figure onto the $x$-axis, so $|J_i|/|J_\emptyset|=2^{-T_i}$. We note that outside $J_1\cup J_2\cup J_3$, $f$ cannot take the value $0$ in view of the zero criterion.

This process continues. Each interval $J_i$ generates, with probability one, three random nonoverlapping subintervals $J_{i1}, J_{i2}$ and $J_{i3}$ above (or below) which the basic figure reappears for the second time in the construction of $f$, etc. This way, we obtain a ternary tree $\{J_\sigma: \sigma\in\Sigma^*\}$ of intervals, where $\Sigma:=\{1,2,3\}$ and $\Sigma^*:=\bigcup_{n=0}^\infty \Sigma^n$. The contraction ratios $R_{\sigma*i}:=|J_{\sigma*i}|/|J_\sigma|$, $\sigma\in\Sigma^*$, $i=1,2,3$ are all independent. If $i\in\{1,2\}$, $R_{\sigma*i}$ has the distribution of $2^{-(1+\tau_1)}$, while $R_{\sigma*3}$ has the distribution of $2^{-\tau_2}$.
Let $K:=\bigcap_{n=1}^\infty \bigcup_{\sigma\in \Sigma^n}J_\sigma$ be the limit set. A by now familiar argument shows that $K=L_f(0)\cap[0,1/2]$. It follows from Theorem 1.1 of Mauldin and Williams \cite{MW} that with probability one, $\dim_H K$ is the unique number $s$ such that $\rE(R_1^s+R_2^s+R_3^s)=1$, in other words, the unique $s$ such that
\begin{equation}
2\rE\left(2^{-s(1+\tau_1)}\right)+\sE\left(2^{-s\tau_2}\right)=1.
\label{eq:random-Moran-equation}
\end{equation}
Now let $\psi_i(x):=\sE(x^{\tau_i})$ denote the probability generating function of $\tau_i$. From standard random walk theory, we have $\psi_2(x)=(\psi_1(x))^2$, and $\psi_1(x)=(1-\sqrt{1-x^2})/x$. Putting $r=2^{-s}$, \eqref{eq:random-Moran-equation} therefore becomes
\begin{equation*}
2r\psi_1(r)+(\psi_1(r))^2=1.
\end{equation*}
This equation has only one positive solution, given by $r^2=(\sqrt{5}-1)/2$. Thus,
\begin{equation*}
s=\frac{-\log r}{\log 2}=\frac{\log(r^{-2})}{\log 4}=\frac{\log\big((1+\sqrt{5})/2\big)}{\log 4}=d_0,
\end{equation*}
and the proof is complete.
\end{proof}

\begin{proof}[Proof of Proposition \ref{prop:Gray-Takagi-zero-set}]
The Gray Takagi function, defined by \eqref{eq:Gray-Takagi}, satisfies the functional equation
\begin{equation}
f(x)=\begin{cases}
x+\frac12 f(2x), & 0\leq x\leq 1/2\\
1-x-\frac12 f(2-2x), & 1/2\leq x\leq 1.
\end{cases}
\label{eq:Gray-Takagi-FE}
\end{equation}
Let $x_m:=1-\sum_{i=1}^m 4^{-(2i-1)}$ for $m\in\ZZ_+$, and let $x^*:=\lim_{m\to\infty}x_m$. Applying \eqref{eq:Gray-Takagi-FE} repeatedly it may be seen that for each $m\in\NN$,
\begin{equation}
x\in[x_m,x_{m-1}] \quad\Longrightarrow \quad f(x)=-4^{-(2m-1)}f\left(4^{2m-1}(x_{m-1}-x)\right). 
\label{eq:Gray-self-similar}
\end{equation}
(The somewhat cumbersome algebraic details are omitted here, but \eqref{eq:Gray-self-similar} is best understood graphically.) Furthermore, $f(x)>0$ for $0<x<x^*$. As a result, $L_f(0)$ consists of $0$, $x^*$, and an infinite sequence of nonoverlapping similar copies of $L_f(0)$ itself, the $m$th copy being scaled by $x_{m-1}-x_m=4^{-(2m-1)}$ and reflected left-to-right. Thus, $\dim_H L_f(0)$ is the unique solution $s$ of the Moran equation
\begin{equation*}
\sum_{m=1}^\infty 4^{-(2m-1)s}=1.
\end{equation*}
A routine calculation gives $s=d_0$.
\end{proof}

In the setting of Model 1, the theorem of Mauldin and Williams is not applicable because the random contraction vectors $(R_{\sigma*1},R_{\sigma*2},R_{\sigma*3})$, $\sigma\in\Sigma^*$ are neither independent nor identically distributed. We must therefore work considerably harder just to obtain the inequalities of Theorem \ref{thm:zero-set-rigid}.

\begin{proof}[Proof of Theorem \ref{thm:zero-set-rigid}(i)]
Assume without loss of generality that $p\geq 1/2$. Define random index sets
\begin{align}
\begin{split}
\Gamma_n&:=\{j: 0\leq j<2^{2n}, f_{2n}(x)=0\ \mbox{for some}\ x\in I_{2n,j}\},\\
\Gamma_n^+&:=\{j\in\Gamma_n: f_{2n}\geq 0\ \mbox{on $I_{2n,j}$ and $s_{2n,j}\neq 0$}\},\\
\Gamma_n^-&:=\{j\in\Gamma_n: f_{2n}\leq 0\ \mbox{on $I_{2n,j}$ and $s_{2n,j}\neq 0$}\}.
\end{split}
\label{eq:random-index-sets}
\end{align}
Let $K_n:=\bigcup_{j\in\Gamma_n}I_{2n,j}$ and $K:=\bigcap_{n=0}^\infty K_n$. Then $K=L_f(0)$ by the zero criterion. Let $\GG_n$ denote the restriction of the graph of $f_{2n}$ to $K_n$, and for an interval $I\subset [0,1]$, let $\GG_n\rest_I$ denote the further restriction of $\GG_n$ to $K_n\cap I$. Suppose $\Gamma_n$ contains three successive integers $j, j+1$ and $j+2$, and let $J:=I_{2n,j}\cup I_{2n,j+1}\cup I_{2n,j+2}$. If the slopes of $f_{2n}$ on the intervals $I_{2n,j}, I_{2n,j+1}$ and $I_{2n,j+2}$ are $2,0,2$ or $-2,0,-2$, respectively, we call $\GG_n\rest_J$ a {\em $Z$-shape}. If instead the slopes are $-2,0,2$, we call $\GG_n\rest_J$ a {\em cup-shape}. Let $N_n$ denote the total number of $Z$-shapes contained in $\GG_n$. We shall show that, given that $N_{n_0}>0$ for some $n_0$, $N_n$ grows at an exponential rate with probability one.

Assume $N_{n_0}>0$. Figure \ref{fig:Z-shape-and-cup-shape} shows the four possible transitions from $f_{2n}$ to $f_{2n+2}$, depending on the signs $\omega_{2n}$ and $\omega_{2n+1}$. The top part of Figure \ref{fig:Z-shape-and-cup-shape} makes clear that if $\GG_n\rest_J$ is a $Z$-shape, then $\GG_{n+1}\rest_J$ contains again a $Z$-shape, at $1/4$ the scale of the original one. Hence, $N_{n+1}\geq N_n$. Moreover, if $(\omega_{2n},\omega_{2n+1})=(1,-1)$, then $\GG_{n+1}\rest_J$ contains in addition a cup-shape. As the bottom part of Figure \ref{fig:Z-shape-and-cup-shape} shows, a cup-shape in $\GG_{n+1}$ induces a pair of $Z$-shapes in $\GG_{n+2}$ if $\omega_{2n+2}=-1$. Thus, a $Z$-shape in $\GG_n$ induces three $Z$-shapes in $\GG_{n+2}$ if $(\omega_{2n},\omega_{2n+1},\omega_{2n+2})=(1,-1,-1)$. By symmetry, the same is true if $(\omega_{2n},\omega_{2n+1},\omega_{2n+2})=(-1,1,1)$. As a result, $\rP(N_{n+2}\geq 3N_n)\geq pq^2+qp^2=pq$. It now follows that for each $m\in\NN$, $N_{n_0+2m}\geq N_{n_0}\tilde{N}_{m}$, where $\tilde{N}_m$ is the product of $m$ independent random variables $Y_1,\dots,Y_m$ each having the distribution $\rP(Y_i=3)=pq=1-\sP(Y_i=1)$. By the strong law of large numbers,
\begin{equation*}
\frac{\log\tilde{N}_m}{m}\to \sE(\log Y_1)=pq\log 3 \qquad\mbox{a.s.},
\end{equation*}
and since $N_n$ is nondecreasing, this implies that given $N_{n_0}>0$,
\begin{equation*}
\liminf_{n\to\infty}\frac{\log N_n}{n}\geq \frac{pq\log 3}{2} \qquad\mbox{a.s.}
\end{equation*}

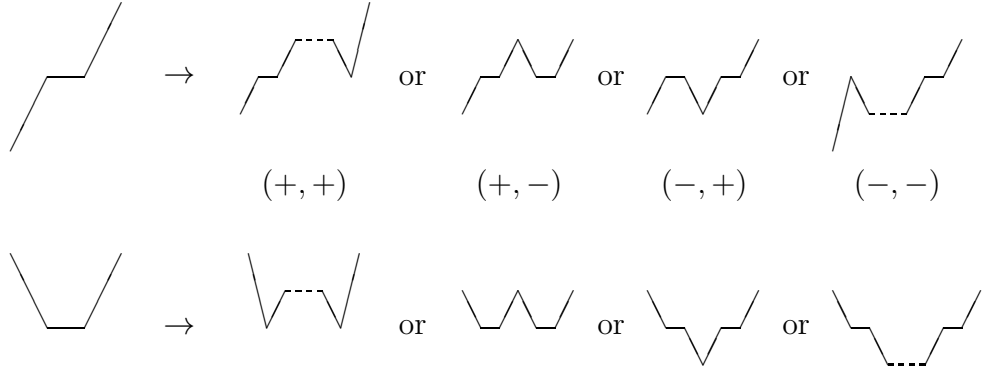
\begin{figure}
\begin{center}
\begin{picture}(360,140)(-10,0)
\put(4,110){\line(-1,-2){14}}
\put(4,110){\line(1,0){14}}
\put(18,110){\line(1,2){14}}
\put(48,112){\makebox(0,0)[tl]{$\to$}}
\put(84,110){\line(-1,-2){7}}
\put(84,110){\line(1,0){7}}
\put(91,110){\line(1,2){7}}
\multiput(98,124)(4,0){4}{\line(1,0){2}}
\put(119,110){\line(-1,2){7}}
\put(119,110){\line(1,4){7}}
\put(137,112){\makebox(0,0)[tl]{or}}
\put(168,110){\line(-1,-2){7}}
\put(168,110){\line(1,0){7}}
\put(175,110){\line(1,2){7}}
\put(189,110){\line(-1,2){7}}
\put(189,110){\line(1,0){7}}
\put(196,110){\line(1,2){7}}
\put(212,112){\makebox(0,0)[tl]{or}}
\put(238,110){\line(-1,-2){7}}
\put(238,110){\line(1,0){7}}
\put(245,110){\line(1,-2){7}}
\put(259,110){\line(-1,-2){7}}
\put(259,110){\line(1,0){7}}
\put(266,110){\line(1,2){7}}
\put(282,112){\makebox(0,0)[tl]{or}}
\put(308,110){\line(-1,-4){7}}
\put(308,110){\line(1,-2){7}}
\multiput(315,96)(4,0){4}{\line(1,0){2}}
\put(336,110){\line(-1,-2){7}}
\put(336,110){\line(1,0){7}}
\put(343,110){\line(1,2){7}}
\put(85,75){\makebox(0,0)[tl]{$(+,+)$}}
\put(166,75){\makebox(0,0)[tl]{$(+,-)$}}
\put(236,75){\makebox(0,0)[tl]{$(-,+)$}}
\put(309,75){\makebox(0,0)[tl]{$(-,-)$}}
\put(4,15){\line(-1,2){14}}
\put(4,15){\line(1,0){14}}
\put(18,15){\line(1,2){14}}
\put(48,17){\makebox(0,0)[tl]{$\to$}}
\put(87,15){\line(-1,4){7}}
\put(87,15){\line(1,2){7}}
\multiput(94,29)(4,0){4}{\line(1,0){2}}
\put(115,15){\line(-1,2){7}}
\put(115,15){\line(1,4){7}}
\put(137,18){\makebox(0,0)[tl]{or}}
\put(168,15){\line(-1,2){7}}
\put(168,15){\line(1,0){7}}
\put(175,15){\line(1,2){7}}
\put(189,15){\line(-1,2){7}}
\put(189,15){\line(1,0){7}}
\put(196,15){\line(1,2){7}}
\put(212,18){\makebox(0,0)[tl]{or}}
\put(238,15){\line(-1,2){7}}
\put(238,15){\line(1,0){7}}
\put(245,15){\line(1,-2){7}}
\put(259,15){\line(-1,-2){7}}
\put(259,15){\line(1,0){7}}
\put(266,15){\line(1,2){7}}
\put(282,18){\makebox(0,0)[tl]{or}}
\put(308,15){\line(-1,2){7}}
\put(308,15){\line(1,0){7}}
\put(315,15){\line(1,-2){7}}
\multiput(322,1)(4,0){4}{\line(1,0){2}}
\put(343,15){\line(-1,-2){7}}
\put(343,15){\line(1,0){7}}
\put(350,15){\line(1,2){7}}
\end{picture}
\end{center}
\caption{The two-step evolution of the $Z$-shape (top) and the cup-shape (bottom)}
\label{fig:Z-shape-and-cup-shape}
\end{figure}

Now fix $\eps>0$. Given that $N_{n_0}>0$, there is with probability one an integer $n_1$ such that
\begin{equation}
n\geq n_1 \quad\Rightarrow \quad \frac{\log N_n}{n}>\frac{pq\log 3}{2}-\eps.
\label{eq:Z-shape-lower-bound}
\end{equation}
Choose $\delta>0$ so that $\delta<4^{-n_1}$, and let $\UU$ be a covering of $K=L_f(0)$ by intervals of length less than $\delta$. By a standard argument, we may assume that $\UU\subset \{I_{2n,j}: n\geq n_1, 0\leq j<2^{2n}\}$. Moreover, we may assume that if $I_{2n,j}\in\UU$, then $j\in\Gamma_n$. There certainly is a smallest $n$ such that $\UU$ contains some interval $I_{2n,j}$; fix this $n$ for the remainder of the proof. If $|s_{2n,j}|\geq 4$, then $\Gamma_{n+1}$ contains only one of the four integers $4j,\dots,4j+3$ and a more efficient covering is obtained by replacing $I_{2n,j}$ with the corresponding subinterval ($I_{2n+2,4j}$ or $I_{2n+2,4j+3}$). Thus, we may assume that $|s_{2n,j}|\leq 2$ whenever $I_{2n,j}\in\UU$. 

If $j_1$ and $j_2$ are distinct members of $\Gamma_n$ such that $f_{2n}\equiv 0$ on $I_{2n,j_1}\cup I_{2n,j_2}$, then $L_f(0)\cap I_{2n,j_1}$ and $L_f(0)\cap I_{2n,j_2}$ are identical up to translation. Hence we may assume that $\UU$ contains either all of the intervals $I_{2n,j}$ on which $f_{2n}\equiv 0$ ($n$ still fixed), or none of these intervals. Similarly, if $j_1$ and $j_2$ are distinct members of $\Gamma_n^+$ such that $|s_{2n,j_1}|=|s_{2n,j_2}|=2$, then $L_f(0)\cap I_{2n,j_1}$ and $L_f(0)\cap I_{2n,j_2}$ are identical up to translation and, possibly, a reflection. Hence we can assume that $\UU$ contains either all of the intervals $I_{2n,j}$ with $j\in\Gamma_n^+$ and $s_{2n,j}=\pm 2$, or none of these intervals. The analogous statement holds for $j\in\Gamma_n^-$.

Now, since a $Z$-shape contains exactly one line segment of each of the three types considered in the last paragraph, it follows that the number of intervals $I_{2n,j}$ that lie in $\UU$ is at least $N_n$. Hence,
\begin{equation}
\sum_{I\in\UU} |I|^s\geq N_n\cdot 4^{-ns}
\label{eq:covering-lower-bound}
\end{equation}
for all $s>0$. Taking 
\begin{equation*}
s=\left(\frac{pq\log 3}{2}-\eps\right)/\log 4,
\end{equation*}
we finally obtain from \eqref{eq:Z-shape-lower-bound} and \eqref{eq:covering-lower-bound} that $\sum_{I\in\UU} |I|^s\geq 1$. Since $\eps$ was arbitrary, this implies that
\begin{equation}
\dim_H L_f(0)\geq \frac{pq\log 3}{2\log 4}>0,
\label{eq:HD-lower-bound}
\end{equation}
almost surely given that $N_{n_0}>0$ for some $n_0$.

In order to complete the proof, we must show that
\begin{equation}
\rP(N_n=0\ \forall\,n)=1-\frac{q}{p}
\label{eq:no-Z-shapes-ever}
\end{equation}
and
\begin{equation}
\rP(\#L_f(0)<\infty|N_n=0\ \forall\,n)=1.
\label{eq:then-finite}
\end{equation}
Let $S_n:=\omega_0+\dots+\omega_{n-1}$. Then $\{S_n\}$ is a simple random walk with parameter $p$, and since $p\geq 1/2$, $S_n$ takes positive values infinitely often with probability one, but will visit the value $-1$ only with probability $q/p$. Observe that $S_n=s_{n,0}$. Suppose that for some integer $n$, $S_{2n-1}=-1$, $S_{2n}=0$ and $S_{2n+1}=1$. Then $s_{2n-1,0}=-1$; $s_{2n,0}=0$ and $s_{2n,1}=-2$; and finally, $s_{2n+1,0}=1$ and $s_{2n+1,1}=s_{2n+1,2}=-1$. Now one sees that, regardless the value of $\omega_{2n+1}$, a $Z$-shape appears in $\GG_{n+1}$ (part of the graph of $f_{2n+2}$) somewhere on $I_{2n+1,1}\cup I_{2n+1,2}$. On the other hand, if $S_n\geq 0$ for every $n$, then $f_n\geq 0$ for each $n$ and no $Z$-shape ever appears. Thus,
\begin{equation*}
\rP(N_n>0\ \mbox{for some $n$})=\sP(\exists\,n: S_{2n-1}=-1, S_{2n}=0, S_{2n+1}=1)=q/p,
\end{equation*}
by the properties of the random walk. This gives \eqref{eq:no-Z-shapes-ever}.

Finally, note that the probability that $S_n\geq 0$ for all $n$ and $S_n=0$ for infinitely many $n$ is zero. Thus,
\begin{equation*}
\rP(S_n>0\ \mbox{for all but finitely many $n$}|N_n=0\ \forall\,n)=1.
\end{equation*}
But if $S_n\geq 0$ for all $n$ and $S_n=0$ for only finitely many $n$, then $\#\Gamma_n$ is eventually constant. More precisely, there is $n_0$ such that for every $n\geq n_0$, $\Gamma_{n+1}$ contains exactly one integer from $\{4j,\dots,4j+3\}$, for each $j\in \Gamma_n$. This clearly implies that the limit set $K=L_f(0)$ is finite. This proves \eqref{eq:then-finite}, and completes the proof of part (i) of the theorem.
\end{proof}

\begin{remark}
{\rm
The lower estimate \eqref{eq:HD-lower-bound} for $\dim_H L_f(0)$ can be improved by considering the ratio $N_{n+k}/N_n$ for values of $k$ larger than $2$, because there is a variety of ways for additional $Z$-shapes to appear in $\GG_{n+k}$, and this number increases exponentially as $k$ increases. The details are cumbersome, however. Using a computer program with the value $k=9$, the author has been able to establish, for the case $p=1/2$, that $\dim_H L_f(0)\geq .2548$ a.s. This is still well below the upper bound of $d_0\approx .3471$. The lower bounds appear to converge extremely slowly as $k\to\infty$, and it is merely a guess that they converge to $d_0$.
}
\end{remark}

\begin{proof}[Proof of Theorem \ref{thm:zero-set-rigid}(ii)]
Assume Model 1 with $p=1/2$. Fix, for the time being, an integer $k\geq 3$. Recall the random index sets defined by \eqref{eq:random-index-sets}, and define random variables
\begin{align*}
c_n&:=\#\{j\in\Gamma_n: f_{2n}\equiv 0\ \mbox{on}\ I_{2n,j}\},\\
u_n^{(i)}&:=\#\{j\in\Gamma_n^+: |s_{2n,j}|=2i\}, \qquad 1\leq i<k,\\
l_n^{(i)}&:=\#\{j\in\Gamma_n^-: |s_{2n,j}|=2i\}, \qquad 1\leq i<k,\\
u_n^{(k)}&:=\#\{j\in\Gamma_n^+: |s_{2n,j}|\geq 2k\},\\
l_n^{(k)}&:=\#\{j\in\Gamma_n^-: |s_{2n,j}|\geq 2k\}.
\end{align*}
The dynamics of these sequences of random variables depend on four cases regarding the signs $(\omega_{2n},\omega_{2n+1})$, as follows.

\bigskip
\begin{tabular}{llll}
$(+,+):$ & $c_{n+1}=l_n^{(1)}$ & $(-,-):$ & $c_{n+1}=u_n^{(1)}$\\
& $u_{n+1}^{(1)}=2c_n$ && $u_{n+1}^{(1)}=u_n^{(1)}+u_n^{(2)}$\\
& $l_{n+1}^{(1)}=l_n^{(1)}+l_n^{(2)}$ && $l_{n+1}^{(1)}=2c_n$\\
& $u_{n+1}^{(i)}=u_n^{(i-1)}, \quad 2\leq i<k\quad$ && $u_{n+1}^{(i)}=u_n^{(i+1)}, \quad 2\leq i<k$\\
& $u_{n+1}^{(k)}=u_n^{(k-1)}+u_n^{(k)}$ && $l_{n+1}^{(i)}=l_n^{(i-1)}, \quad 2\leq i<k$\\
& $l_{n+1}^{(i)}=l_n^{(i+1)}, \quad 2\leq i<k$ && $l_{n+1}^{(k)}=l_n^{(k-1)}+l_n^{(k)}$\\
& $l_{n+1}^{(k)}\leq l_n^{(k)}$ && $u_{n+1}^{(k)}\leq u_n^{(k)}$
\end{tabular}

\bigskip
\begin{tabular}{llll}
$(+,-)$: & $c_{n+1}=2c_n$ & $(-,+)$: & $c_{n+1}=2c_n$\\
& $u_{n+1}^{(1)}=2c_n+u_n^{(1)}$ && $u_{n+1}^{(i)}=u_n^{(i)}, \quad 1\leq i\leq k$\\
& $u_{n+1}^{(i)}=u_n^{(i)}, \quad 2\leq i\leq k\ \qquad$ && $l_{n+1}^{(1)}=2c_n+l_n^{(1)}$\\
& $l_{n+1}^{(i)}=l_n^{(i)}, \quad 1\leq i\leq k$ && $l_{n+1}^{(i)}=l_n^{(i)}, \quad 2\leq i\leq k$
\end{tabular}

\bigskip
\noindent Let ${\bf z}_n:=\big(c_n,u_n^{(1)},\dots,u_n^{(k)},l_n^{(1)},\dots,l_n^{(k)}\big)$, and let $\FF_n$ be the $\sigma$-algebra generated by the random vectors ${\bf z}_1,\dots,{\bf z}_n$. Let $\sigma_n^{(i)}:=u_n^{(i)}+l_n^{(i)}$, $i=1,\dots,k$. Since $p=1/2$, the four cases above all occur with probability $1/4$, and hence we have
\begin{align}
\begin{split}
\rE(c_{n+1}|\FF_n)&=c_n+\frac14 \sigma_n^{(1)},\\
\rE(\sigma_{n+1}^{(1)}|\FF_n)&=2c_n+\frac34 \sigma_n^{(1)}+\frac14 \sigma_n^{(2)},\\
\rE(\sigma_{n+1}^{(i)}|\FF_n)&=\frac14 \sigma_n^{(i-1)}+\frac12 \sigma_n^{(i)}+\frac14 \sigma_n^{(i+1)}, \qquad 2\leq i<k,\\
\rE(\sigma_{n+1}^{(k)}|\FF_n)&\leq \frac14 \sigma_n^{(k-1)}+\sigma_n^{(k)}.
\end{split}
\label{eq:conditional-expectations}
\end{align}
Put ${\bf x}_n:={\big(c_n,\sigma_n^{(1)},\dots,\sigma_n^{(k)}\big)}^t$. By \eqref{eq:conditional-expectations},
\begin{equation}
\rE({\bf x}_{n+1}|\FF_n)\leq A_k {\bf x}_n, 
\label{eq:conditional-domination}
\end{equation}
where $A_k$ is the $(k+1)\times(k+1)$ tridiagonal matrix 
\begin{equation*}
A_k:=\begin{bmatrix}
1 & 1/4 & 0 & \dots & \dots & 0\\
2 & 3/4 & 1/4 & \ddots && \vdots\\
0 & 1/4 & 1/2 & 1/4 & \ddots & \vdots\\
\vdots & \ddots & \ddots & \ddots & \ddots & 0\\
\vdots & & \ddots & 1/4 & 1/2 & 1/4\\
0 & \dots & \dots & 0 & 1/4 & 1
\end{bmatrix}.
\end{equation*}
Let $\rho_k$ denote the spectral radius of $A_k$. Since $A_k$ is nonnegative $\rho_k$ is an eigenvalue of $A_k$, and since $A_k$ is irreducible, the Perron-Frobenius theorem guarantees the existence of a positive left eigenvector ${\bf v}_k$ of $A_k$ corresponding to $\rho_k$. It follows by \eqref{eq:conditional-domination} that the process
\begin{equation}
X_n:=\rho_k^{-n}{\bf v}_k{\bf x}_n, \qquad n\in\ZZ_+
\label{eq:supermartingale}
\end{equation}
is a positive supermartingale, which by the Martingale Convergence Theorem converges almost surely to a finite nonnegative limit $X_\infty$. Let $\delta$ be the smallest entry of the vector ${\bf v}_k$. Then $\delta>0$, and for a given $\eps>0$, \eqref{eq:supermartingale} implies that for all sufficiently large $n$,
\begin{equation*}
c_n+\sum_{i=1}^k \sigma_n^{(i)} \leq \delta^{-1}\rho_k^n X_n \leq \delta^{-1}(X_\infty+\eps)\rho_k^n.
\end{equation*}
Thus, by the zero criterion, the number of intervals $I_{2n,j}$ needed to cover $L_f(0)$ grows at most at rate $\rho_k^n$. Consequently,
\begin{equation*}
\dim_H L_f(0)\leq \frac{\log\rho_k}{\log 4} \quad \mbox{a.s.}
\end{equation*}
The proof will be complete if we can show that
\begin{equation}
\liminf_{k\to\infty} \rho_k\leq \frac{1+\sqrt{5}}{2}.
\label{eq:liminf-inequality}
\end{equation}
Let $\tilde{A}_k$ be the $k\times k$ matrix obtained by deleting the last row and last column of $A_k$. For completeness, we define $\tilde{A}_1=[1]$ and $\tilde{A}_2=\begin{bmatrix} 1 & 1/4\\2 & 3/4\end{bmatrix}$. Let
\begin{equation*}
\xi_k(\lambda):=\det(A_k-\lambda I_{k+1}), \qquad \zeta_k(\lambda):=\det(\tilde{A}_k-\lambda I_{k})
\end{equation*}
be the characteristic polynomials of $A_k$ and $\tilde{A}_k$, respectively. Then
\begin{equation}
\xi_k(\lambda)=(1-\lambda)\zeta_k(\lambda)-\frac{1}{16}\zeta_{k-1}(\lambda),
\label{eq:xi-and-zeta}
\end{equation}
and $\zeta_k$ satisfies the recursion
\begin{equation}
\zeta_k(\lambda)=\left(\frac12-\lambda\right)\zeta_{k-1}(\lambda)-\frac{1}{16}\zeta_{k-2}(\lambda), \qquad k\geq 3,
\label{eq:zeta-recursion}
\end{equation}
with initial conditions
\begin{equation}
\zeta_1(\lambda)=1-\lambda, \qquad \zeta_2(\lambda)=\lambda^2-\frac74\lambda+\frac14.
\label{eq:zeta-initial}
\end{equation}
For fixed real $\lambda>1$, the solution of the system \eqref{eq:zeta-recursion}, \eqref{eq:zeta-initial} can be written as
\begin{equation}
\zeta_k(\lambda)=C_1(\lambda)\gamma_1^{k-1}(\lambda)+C_2(\lambda)\gamma_2^{k-1}(\lambda),
\label{eq:zeta-solution}
\end{equation}
where
\begin{equation}
\gamma_1(\lambda):=\frac12\left(\frac12-\lambda-\sqrt{\lambda^2-\lambda}\right), \qquad
\gamma_2(\lambda):=\frac12\left(\frac12-\lambda+\sqrt{\lambda^2-\lambda}\right),
\label{eq:gammas}
\end{equation}
and
\begin{equation}
C_1(\lambda):=\frac{\zeta_1(\lambda)\gamma_2(\lambda)-\zeta_2(\lambda)}{\gamma_2(\lambda)-\gamma_1(\lambda)}, \qquad 
C_2(\lambda):=\frac{\zeta_2(\lambda)-\zeta_1(\lambda)\gamma_1(\lambda)}{\gamma_2(\lambda)-\gamma_1(\lambda)}.
\label{eq:zeta-coefficients}
\end{equation}
Substituting \eqref{eq:zeta-solution} into \eqref{eq:xi-and-zeta} and some rearranging finally gives
\begin{equation}
\xi_k(\lambda)=\tilde{C}_1(\lambda)\gamma_1^{k-2}(\lambda)+\tilde{C}_2(\lambda)\gamma_2^{k-2}(\lambda), \qquad k\geq 3,
\label{eq:xi-solution}
\end{equation}
for real $\lambda>1$, where
\begin{equation}
\tilde{C}_i(\lambda):=\left((1-\lambda)\gamma_i(\lambda)-\frac{1}{16}\right)C_i(\lambda), \qquad i=1,2.
\label{eq:tilde-constants}
\end{equation}
Now observe that $\lambda>1$ implies $|\gamma_2(\lambda)|\leq 1/4$ and $\gamma_1(\lambda)<0$, with $|\gamma_1(\lambda)|>1$ if and only if $\lambda>25/16$. If $\liminf\rho_k\leq 25/16$ we are done, since $25/16=1.5625<(1+\sqrt{5})/2$. So assume $\liminf\rho_k>25/16$; then we can find $\eta>0$ such that $\rho_k\geq 25/16+\eta$ for all sufficiently large $k$, and hence $|\gamma_1(\rho_k)|\geq 1+\eps$ for all sufficiently large $k$, for some $\eps>0$ depending on $\eta$. Since $\xi_k(\rho_k)=0$ for all $k$, this, together with \eqref{eq:xi-solution}, implies that $\tilde{C}_1(\rho_k)$ must tend to zero as $k\to\infty$. Hence, by \eqref{eq:tilde-constants}, either $C_1(\rho_k)\to 0$ or $(1-\rho_k)\gamma_1(\rho_k)\to 1/16$. It is easy to see that the latter is impossible, and so $\rho_*:=\liminf\rho_k$ must be a solution of $C_1(\rho_*)=0$, as $C_1(\lambda)$ is continuous in $\lambda$. Routine algebra using \eqref{eq:zeta-initial}, \eqref{eq:gammas} and \eqref{eq:zeta-coefficients} shows that the only solution of $C_1(\rho_*)=0$ is $\rho_*=(1+\sqrt{5})/2$. This proves \eqref{eq:liminf-inequality}, completing the proof of the theorem.
\end{proof}

\subsection{Dimension of the maximum set}

\begin{proof}[Proof of Theorem \ref{thm:maximal-set-flexible}]
Assume Model 2. Slightly abusing notation, put $M_n:=\max\{f_n(x): 0\leq x\leq 1\}$ and $\MM_n:=\{x\in[0,1]: f_n(x)=M_n\}$. Note, since $\phi(x)+(1/2)\phi(2x)\leq 1/2$, that
\begin{equation*}
M_{2n}\leq \sum_{i=0}^{n-1}\frac12\left(\frac14\right)^i, \qquad n\in\ZZ_+. 
\end{equation*}
Define random index sets
\begin{equation*}
\Lambda_n^1:=\left\{j: 0\leq j<2^{2n}, f_{2n}\equiv\sum_{i=0}^{n-1}\frac12{\left(\frac14\right)}^i\ \mbox{on}\ I_{2n,j}\right\}, \quad n\in\ZZ_+,
\end{equation*}
and put $X_n^1:=\#\Lambda_n^1$. Then $\{X_n^1\}$ is a Galton-Watson (GW) process with initial value $X_0^1=1$ and offspring distribution $\boldsymbol{\pi}=(\pi_0,\pi_1,\pi_2)=(1-2p^2 q-p^3,2p^2 q,p^3)$, where $q=1-p$. The offspring distribution has mean $\mu:=\sum_{i=0}^2 i\pi_i=2p^2$ and probability generating function $h(t):=\sum_{i=0}^2\pi_i t^i=(1-2p^2 q-p^3)+2p^2 qt+p^3 t^2$. Let $\rho:=\sP(X_n^1\to 0)$. According to the basic theory of the GW process (e.g. \cite{Athreya}), $\rho=1$ when $\mu\leq 1$, and in that case, $M_f<2/3$ with probability one. Assume from now on that $\mu>1$; that is, $p>1/\sqrt{2}$. Then $\rho$ is the smallest positive number satisfying $t=h(t)$, so that $\rho=(1-2p^2 q-p^3)/p^3$, and
\begin{equation*}
\rP(M_f=2/3)=\sP(X_n^1>0\ \forall\, n)=1-\rho=\frac{2p^2-1}{p^3}.
\end{equation*}
This establishes part (i) of the theorem. Next, define the random set
\begin{equation*}
F_1:=\bigcap_{n=0}^\infty	\bigcup_{j\in\Lambda_n^1}I_{2n,j}.
\end{equation*}
Then $F_1=\emptyset$ if and only if $X_n^1\to 0$, and given that $F_1\neq\emptyset$, $\dim_H F_1=\log\mu/\log 4$ a.s. by Theorem 1.1 of Mauldin and Williams \cite{MW}.

Put $\tau_0\equiv 0$. Proceeding inductively, suppose processes $\{X_n^1\},\dots,\{X_n^k\}$ and random variables $\tau_0,\dots,\tau_{k-1}$ have been defined and that $X_n^r\to 0$ for $r=1,\dots,k$. Let $N_k:=\min\{n: X_n^k=0\}$, and define
\begin{equation*}
\tau_k:=\inf\{n\geq\tau_{k-1}+N_k: \MM_{2n}\ \mbox{contains an interval}\}.	
\end{equation*}
Since $p\geq 1/2$, $\tau_k$ is finite almost surely. (Put $m=\tau_{k-1}+N_k$, and let $x_0:=j/2^{2m}$ be a point of maximum of $f_{2m}$. The slope of $f_{2m+n}$ directly to the right of $x_0$ starts with a nonpositive value and follows (as a function of $n\in\ZZ_+$) a simple random walk with parameter $p$, so it will eventually reach $0$.) Note that
\begin{equation*}
M_{2(\tau_k+n)}\leq M_{2\tau_k}+\sum_{i=0}^{n-1}\frac12{\left(\frac14\right)}^{\tau_k+i}, \qquad n\in\ZZ_+.
\end{equation*}
Define the random index sets
\begin{equation*}
\Lambda_n^{k+1}:=\left\{j: 0\leq j<2^{2(\tau_k+n)}, f_{2(\tau_k+n)}\equiv M_{2\tau_k}+\sum_{i=0}^{n-1}\frac12{\left(\frac14\right)}^{\tau_k+i}\ \mbox{on}\ I_{2(\tau_k+n),j}\right\},
\end{equation*}
for $n\in\ZZ_+$, and put $X_n^{k+1}:=\#\Lambda_n^{k+1}$. By definition of $\tau_k$, $\Lambda_0^{k+1}\neq\emptyset$ and so $X_0^{k+1}\geq 1$. Now $\{X_n^{k+1}\}$ is again a GW process with offspring distribution $\boldsymbol{\pi}$, and it depends on the preceding processes $\{X_n^1\},\dots,\{X_n^k\}$ only through the value of $X_0^{k+1}$. Thus,
\begin{align}
\begin{split}
\rP(X_n^{k+1}\to & \ 0\,|X_n^1\to 0,\dots,X_n^k\to 0)\\
&\leq \sP(X_n^{k+1}\to 0\,|X_n^1\to 0,\dots,X_n^k\to 0, X_0^{k+1}=1)=\rho<1.
\end{split}
\label{eq:again-extinct}
\end{align}
Define the random set
\begin{equation*}
F_{k+1}:=\bigcap_{n=0}^\infty	\bigcup_{j\in\Lambda_n^{k+1}}I_{2(\tau_k+n),j}.
\end{equation*}
Then $F_{k+1}=\emptyset$ if and only if $X_n^{k+1}\to 0$ as $n\to\infty$, and given that $X_n^{k+1}>0$ for all $n$, $\dim_H F_{k+1}=\log\mu/\log 4$ a.s. 

Now \eqref{eq:again-extinct} implies that with probability one, there will eventually be a $k\in\NN$ such that $X_n^k>0$ for all $n$, and for that $k$, we have $F_k=\MM_f$, and
\begin{equation*}
M_f=M_{2\tau_{k-1}}+\sum_{i=0}^\infty \frac12\left(\frac14\right)^{\tau_{k-1}+i} = M_{2\tau_{k-1}}+\frac23\left(\frac14\right)^{\tau_{k-1}}.
\end{equation*}
Part (ii) of the theorem now follows.
\end{proof}

\section{Open problems} \label{sec:open}

There are many natural questions left to answer. A few are listed here.

\medskip
\noindent {\bf Problem 1.} Does there exist $f\in\TT_v$ such that $\dim_H L_f(y)=0$ for every $y\in\RR$?

\medskip
\noindent {\bf Problem 2.} Is it true for all $f\in\TT_v$ that $L_f(y)$ is finite for Lebesgue-almost every $y$? Is this true for the Gray Takagi function?

\medskip
\noindent {\bf Problem 3.} (Random case, Model 1) Prove or disprove that $\dim_H L_f(0)=d_0$ a.s. when $p=1/2$.

\medskip
\noindent {\bf Problem 4.} (Random case, Model 2) What is the best (smallest) bound $s_1$ such that, for each $y\in\RR$, $\rP(\dim_H L_f(y)\leq s_1)=1$? More strongly, what is the smallest $s_2$ such that
$\rP(\dim_H L_f(y)\leq s_2\ \forall\,y\in\RR)=1$? (Obviously, $s_1\leq s_2\leq d_v^*$.) In the case of Model 1, only the first question is of interest, as $s_2=1/2$ in view of Theorem \ref{thm:rigid-upper-bound}.

\medskip
\noindent {\bf Problem 5.} (Random case, Model 2) Prove or disprove that $\dim_H \MM_f=0$ a.s. when $p\leq 1/\sqrt{2}$, and that $\MM_f$ is finite a.s. when $p<1/2$. What can one say about the distribution of $M_f$ in these cases?

\section*{Acknowledgments}
The author is grateful to the referee for a careful reading of the manuscript and for suggesting improvements to the presentation of the paper.

\end{document}